\pdfoutput=1
\documentclass[a4paper]{article}

\usepackage{amsmath}
\usepackage{amssymb}
\usepackage{amsthm}

\usepackage[backend=bibtex,bibstyle=alphabetic,citestyle=alphabetic]{biblatex} 
	\AtEveryBibitem{\clearfield{doi}}
	\AtEveryBibitem{\clearfield{isbn}}
	\AtEveryBibitem{\clearfield{issn}}
	\AtEveryBibitem{\clearfield{url}}

\usepackage{dsfont}
\usepackage{enumerate}
\usepackage{pifont}
\usepackage{theoremref}

\newcommand{\mc}[1]{{\mathcal{#1}}}
\newcommand{\mf}[1]{{\mathfrak{#1}}}
\newcommand{\bb}[1]{{\mathbb{#1}}}

\DeclareMathOperator{\RE}{Re}
\DeclareMathOperator{\IM}{Im}
\renewcommand{\Re}{\RE}
\renewcommand{\Im}{\IM}

\DeclareMathOperator{\spr}{spr}
\DeclareMathOperator{\reg}{reg}
\DeclareMathOperator{\Cl}{\text{\it{c}}\hspace{-0.8pt}\ell\hspace{-0.2pt}}
\DeclareMathOperator{\dom}{dom}
\DeclareMathOperator{\mul}{mul}
\DeclareMathOperator{\ran}{ran}

\newcommand{\smm}[1][\alpha & \beta \\ \gamma & \delta]{\bigl(\begin{smallmatrix} #1 
\end{smallmatrix}\bigr)}

\newlength{\maxlabwidth}

\numberwithin{equation}{section}
\swapnumbers
\theoremstyle{plain}
	\newtheorem{lemma}{Lemma}[section]
	\newtheorem{proposition}[lemma]{Proposition}
	\newtheorem{theorem}[lemma]{Theorem}
	\newtheorem{corollary}[lemma]{Corollary}
	\newtheorem{ntheoreM}[lemma]{}
\theoremstyle{definition}
	\newtheorem{definitioN}[lemma]{Definition}
\theoremstyle{remark}
	\newtheorem{remarK}[lemma]{Remark}
	\newtheorem{examplE}[lemma]{Example}
	\newtheorem{nremarK}[lemma]{}
\newcommand{\thlab}[1]{\thlabel{#1}\label{#1.}}
\renewcommand{\qedsymbol}{\raisebox{-2pt}{\large\ding{113}}}
\newcommand{\defendsymbol}{$\lozenge$}
\newcommand{\qedsymbolsave}{\qedsymbol}

\newenvironment{definition}{\begin{definitioN}}{
	\renewcommand{\qedsymbolsave}{\qedsymbol}\renewcommand{\qedsymbol}{\defendsymbol}
	\popQED{\qed}\renewcommand{\qedsymbol}{\qedsymbolsave}\end{definitioN}}
\newenvironment{remark}{\begin{remarK}}{
	\renewcommand{\qedsymbolsave}{\qedsymbol}\renewcommand{\qedsymbol}{\defendsymbol}
	\popQED{\qed}\renewcommand{\qedsymbol}{\qedsymbolsave}\end{remarK}}
\newenvironment{example}{\begin{examplE}}{
	\renewcommand{\qedsymbolsave}{\qedsymbol}\renewcommand{\qedsymbol}{\defendsymbol}
	\popQED{\qed}\renewcommand{\qedsymbol}{\qedsymbolsave}\end{examplE}}

\begin{document}
\begin{flushleft}
	{\Large\bf Functional Calculus for definitizable self-adjoint linear relations on Krein spaces}
	\\[5mm]
	\textsc{Michael Kaltenb\"ack, Raphael Pruckner\footnote{This work was supported by a joint project of the Austrian Science Fund 
		  (FWF, I1536--N25) and the Russian Foundation for Basic Research (RFBR, 13-01-91002-ANF).}}
	\\[5mm] 
	\textit{\small Dedicated to Heinz Langer on the occasion of his 80th birthday.} 
	\\[6mm]
	{\small
	\textbf{Abstract}: In the present note a functional calculus $\phi \mapsto \phi(A)$ for self-adjoint 
		definitizable linear relations on Krein spaces is developed. This functional calculus is the 
		proper analogue of $\phi \mapsto \int \phi \, dE$ in the Hilbert space situation
		where $\phi$ is a bounded and measurable function on $\sigma(A)$
		and $\int \phi \, dE$ is defined in the weak sense. The derived functional calculus also 
		comprises the Spectral Theorem for self-adjoint definitizable operators on Krein spaces
		showing the existence of spectral projections.
	}
\end{flushleft}

\begin{flushleft}
\end{flushleft}
\begin{flushleft}
\end{flushleft}

\section{Introduction}

In this paper we present a new access to the spectral theorem for 
definitizable, self-adjoint linear relations on Krein spaces $\mc K$.
The spectral theorem for definitizable, self-adjoint operators $A$ with non-empty resolvent sets 
was first shown by Heinz Langer (see \cite{langer1982}). 
He used Riesz projections in order to reduce the problem to the case that
the spectrum of $A$ is contained in $\bb R$. Then, certain
Cauchy-type integrals gave the desired spectral projectors $F(\Delta)$.

In the paper \cite{jonas1981} by Peter Jonas, another way was taken in order to derive the spectral theorem.
After getting rid of the non-real spectrum by the same means as Heinz Langer did, 
Peter Jonas considered the rational functional calculus for densely defined operators
with non-empty resolvent sets, and extended this calculus
to the class $C^\infty(\bb R \cup \{\infty\})$ of infinitely often differentiable functions on $\bb R \cup \{\infty\}$.
For intervals $\Delta$ with endpoints not belonging to the set of critical spectral points
of $A$, a monotone approximation procedure for the characteristic function $\mathds{1}_\Delta$
by $C^\infty(\bb R \cup \{\infty\})$-functions was used in order to obtain the spectral projectors $F(\Delta)$.

Later on, using his theory of locally definitizable operators 
Peter Jonas 
extended the existence of the spectral projections to the case that $A$ 
is a definitizable, self-adjoint linear relation with a non-empty resolvent set; see \cite{jonas1986} and
\cite{jonas2003}. Hereby, definitizability means that $[q(A)x,x]\geq 0$ for all $x \in \mc K$ and some
rational function $q$ whose poles are contained in $\rho(A)$. We call $q$ a definitizing rational function for $A$.
$q(A)$ is well-defined by the rational functional calculus for linear relations with non-empty resolvent sets.
Focusing on the existence of spectral projections all versions of the Spectral Theorem for definitizable
linear relations so far are rather technical and contain tedious calculations. 

In the present work we shall derive a version of the spectral theorem for definitizable, self-adjoint linear 
relations $A$ with non-empty resolvent sets in Krein spaces $\mc K$,
which not only gives the existence of spectral projections, but also the possibility to define a
$B(\mc K)$-valued functional calculus very similar to the functional calculus $\phi\mapsto \int \phi \, dE$ 
for self-adjoint operators on Hilbert spaces. Moreover, the methods used here are of 
structural nature and contain much less calculations compared to previous versions.   
The proper family $\mc F(q,A)$ of functions suitable for our functional 
calculus are bounded and measurable functions on $\sigma(A)$, which assume values in $\bb C$ on 
$\sigma(A)\setminus q^{-1}\{0\}$ and values in $\bb C^{\mf d(w)}$ for a certain $\mf d(w)\in \bb N$ on points 
$w\in q^{-1}\{0\} \cap \sigma(A)$. Moreover, the functions from $\mc F(q,A)$ have to satisfy a certain 
regularity condition at points $w\in q^{-1}\{0\} \cap \sigma(A)$ which are accumulation points of $\sigma(A)$.

In the present paper we combine ideas from \cite{jonas1986}, \cite{dritschel1993}, 
\cite{dritschelrovnyak1996} and from the theory of linear relations.
Starting  with a Krein space $\mc K$, a definitizable, self-adjoint linear relation $A$ with a non-empty resolvent set
and a definitizing rational function $q$, we construct a Hilbert space $\mc V$, which is 
the completion of $\mc K$ with respect to $[q(A).,.]$. After this we define $T: \mc V \to \mc K$
via its Krein space adjoint $T^+: \mc K \to \mc V$, $T^+(x):=x + \{y\in\mc K : [q(A)y,y] = 0\}$.
$T: \mc V \to \mc K$ is an injective, bounded linear mapping such that 
$TT^+ = q(A)$. Then, we consider the linear relation 
$\Theta(A):=(T\times T)^{-1}(A)$ on $\mc V$, which turns out to be self-adjoint. Its spectrum coincides with 
$\sigma(A)$ up to a finite subset of $q^{-1}\{0\}$.
Having the Hilbert space version of the spectral theorem in hand, we can consider $T \int g \, dE \ T^+ \in B(\mc K)$
for any bounded and measurable $g:\sigma(A) \to \bb C$, where $E$ is the spectral measure corresponding to $\Theta(A)$.

In \cite{dritschel1993} the author considered something very similar to $T \int g \, dE \ T^+$ in the case that 
$A$ is an operator and he denoted this expression by $q(A)g(A)$. So he derived something like 
a functional calculus for functions of the form
$q g$, where $g$ is bounded and measurable and where $q$ is a definitizing polynomial.

In our setting, any function $\phi$ from $\mc F(q,A)$ can be decomposed as 
$\phi= s + q g$ where $s$ is a rational function with poles contained in $\rho(A)$ and
$g$ is bounded and measurable. It will turn out that then $\phi\mapsto \phi(A) := s(A) + T \int g \, dE T^+$
constitutes a functional calculus. This is the proper analogue
of $\phi\mapsto\int \phi \, dE$ in the Hilbert space situation. 

Since bounded and measurable functions 
$f: \dom(f) \to \bb C$ with $\sigma(A)\subseteq \dom(f)$, 
which are holomorphic locally at $q^{-1}\{0\} \cap \sigma(A)$,
can be considered as elements of $\mc F(q,A)$ in a canonical way, 
$f(A)$ makes sense for a wide class of functions $f$.
For example, for Borel subsets $\Delta$ of $\bb R\cup\{\infty\}$ such that 
$\partial \Delta \cap q^{-1}\{0\} \cap \sigma(A) = \emptyset$,
the spectral projections $F(\Delta)$ derived by Langer and Jonas
are in our notation nothing else but $\mathds{1}_\Delta(A)$.

Moreover, with our access to the spectral theorem for definitizable linear relations, it is no longer necessary to
split off the non-real part of the spectrum. In fact, the non-real spectral points can be treated like the
points from the real part. Using our new access to the spectral theorem for self-adjoint, definitizable linear relations,
it should be possible to develop also a spectral theorem for normal, definitizable linear relations, which will
be the subject of a forthcoming paper.

At the beginning of the present paper, we recall some well known concepts from the theory of linear relations.
First we recall the M\"obius-Calculus together with elementary definitions. Then we shortly recall
the rational functional calculus for linear relations. After this, in Section \ref{movlinrel}, we will study elementary 
properties of linear relations of the form $(T\times T)(B) = TBT^{-1}$ and $(T\times T)^{-1}(A)=T^{-1}AT$, where
$T: \mc V \to \mc K$ is an everywhere defined linear mapping and $A$ ($B$) is a linear relation on $\mc K$ ($\mc V$). 
In Section \ref{movlinrelrein} these studies are continued under the assumptions that $\mc K$ is a Krein space, 
$\mc V$ is a Hilbert space and that $T: \mc V \to \mc K$ is one-to-one. We will see that
$A\mapsto (T\times T)^{-1}(A)$ gives rise to
a $*$-algebra homomorphism $\Theta: (TT^+)' \to (T^+T)'$ where $(TT^+)'\subseteq B(\mc K)$ ($(T^+T)'\subseteq B(\mc V)$) 
denotes the commutant of $TT^+$ ($T^+T$). The mapping $\Xi:D\mapsto TDT^+$ from $(T^+T)'$ into $(TT^+)'$ acts in the 
opposite direction and satisfies $\Xi \circ \Theta(C)=C TT^+$.
In Section \ref{deflinrel}, the derived results are applied to the situation when $\mc V$ and $T$
come from a definitizable linear relations $A$ as indicated above.
In the final section we define the class $\mc F(q,A)$ of functions $\phi$
for which we are able to define $\phi(A)$, and prove in \thref{zwe4578} the main result of the present note. 
Providing $\mc F(q,A)$ with a proper topology, we will see
in \thref{calculstetig} that our functional calculus is continuous.

\section{M\"obius-Calculus for Linear Relations}

First, we would like to recall the M\"obius type transformation of a linear relation as discussed, for instance, in 
\cite{dijksmadesnoo1987b}.

Let $\mc K$ be a vector space. For any $2 \times 2$-matrix $M=\smm[\alpha & \beta \\ \gamma & \delta]\in\mathbb C^{2\times 2}$, 
we define $\tau_M : \mc K\times \mc K \to \mc K\times \mc K$ via its block structure as
\[
    \tau_M := \begin{pmatrix} \delta I & \gamma I \\ \beta I & \alpha I \end{pmatrix} \,,
\]
i.e.\
\[
	\tau_M(x;y) = (\delta x+\gamma y; \beta x + \alpha y) \ \ \text{ for all } \ \ (x;y)\in \mc K\times \mc K \,.
\]

The good thing about this transformation is that 
for a linear relation $A$ on $\mc K$, i.e.\ a linear subspaces of $\mc K \times \mc K$,  
other linear relations like $A-\lambda I$, $\lambda A$, $A^{-1}$ and, more generally, $(\alpha A + \beta I)$, 
$\alpha I + \beta(A+\gamma)^{-1}$
can be expressed as $\tau_M(A)$ with an appropriate matrix $M\in \mathbb C^{2\times 2}$:
\begin{equation}\label{bspfmoeb}
(\alpha A + \beta I) = \tau_{\smm[ \alpha & \beta \\ 0 & 1]}(A), \quad
\alpha + \beta(A+\gamma)^{-1}=  
		\tau_{\smm[ \alpha & \beta + \gamma \alpha \\ 1 & \gamma]}(A) \,.
\end{equation}

We recall from \cite{dijksmadesnoo1987b}:

\begin{lemma}\thlab{moebtrafoeig}
    Let $\mc K$ be a vector space and let $M, N\in\mathbb C^{2\times 2}$. 
    Then $\tau_N\circ\tau_M = \tau_{NM}$.
    If $M\in\mathbb C^{2\times 2}$ is invertible, then $(\tau_M)^{-1}=\tau_{M^{-1}}$.
\end{lemma}

The resolvent set and the spectrum of a linear relation $A$ on a Banach space $\mc K$ are defined almost 
as in the operator case. In fact,
\[
	\rho(A):= \{ \lambda \in \bb C \cup \{ \infty \} : (A-\lambda)^{-1}\in B(\mc K)\},
\]
and, as usual, $\sigma(A):= (\mathbb C \cup \{\infty\}) \backslash\rho(A)$.
Here, $(A-\lambda)^{-1}$ exists always as a linear relation on $\mc K$, and $\lambda\in \rho(A)$ just means that 
this linear relation is actually an everywhere defined and bounded operator on $\mc K$, i.e. an element of $B(\mc K)$. 
Also note that $(A-\infty)^{-1}:=A$. Hence, $\infty\in \rho(A)$ just means that $A$ is a bounded and everywhere
defined operator. Accordingly, we  set  
\[
    \ran (A-\infty):=\dom A = \{x : (x;y) \in A \ \text{ for some } \ y \in \mc K\} \,,
\]
\[
    \dom (A-\infty):=\ran A= \{y : (x;y) \in A \ \text{ for some } \ x \in \mc K\} \,.
\]
A slightly more general concept as the resolvent set is the set of \emph{points of regular type} for $A$,
\[
  \reg(A):=\{ \lambda \in \mathbb C \cup \{\infty\} : (A-\lambda)^{-1} \in B( \ran (A-\lambda)) \} \supseteq \rho(A) \,.
\]

For the following assertion, see \cite{dijksmadesnoo1987b}.
\begin{theorem}\thlab{spectrtrans}
	Let $A$ be a linear relation on a Banach space $\mc K$, and let 
	$M =\smm[\alpha & \beta \\ \gamma & \delta]\in\mathbb C^{2\times2}$ be invertible.
	Then, we have
	\[
		    \reg(\tau_M(A)) = \phi_{M}(\reg(A)),  \ \
				\rho(\tau_M(A)) = \phi_{M}(\rho(A)), \ \
				\sigma(\tau_M(A)) = \phi_{M}(\sigma(A)) \,,
	\] 
	where $\phi_M: \mathbb C \cup \{\infty\} \to \mathbb C \cup \{\infty\}$ 
	denotes the M\"obius transform $z \mapsto \frac{\alpha z+ \beta}{\gamma z + \delta}$ related to $M$. 
\end{theorem}

\section{Rational Functional Calculus}

As the results in this section are more or less folklore and their verification is
straight forward, we shall skip proofs.

Let $A$ be a linear relation on a Banach space $\mc K$.
It is well known that $\rho(A)$ (as well as $r(A)$) is an open subset of $\mathbb C\cup\{\infty\}$, and
for $\mu,\lambda \in \rho(A)\setminus \{\infty\}$ the so-called \emph{resolvent equality} 
(see for example \cite{dijksmadesnoo1974}) 
\[
    (A-\lambda)^{-1} - (A-\mu)^{-1} = (\lambda-\mu)(A-\lambda)^{-1} (A-\mu)^{-1}
\]
holds true. 

For the following, assume that $\rho(A)\neq \emptyset$. By 
$\bb C_{\rho(A)}(z)$ we denote the set of all rational functions with poles in 
$\rho(A) \ (\subseteq \bb C \cup \{\infty\})$. Recall that $\infty$ is a pole of the
rational function $r(z)=\frac{u(z)}{v(z)}$ if the polynomial $u(z)$ is of degree greater than
$v(z)$.

By partial fraction decomposition, any rational function $r(z)$ can be represented as
\begin{equation}\label{fracdec}
    r(z) = p(z) + \sum_{k=1}^m \sum_{j=1}^{n(k)} \frac{c_{kj}}{(z-\alpha_k)^j} \,,
\end{equation}
where $p(z)$ is a polynomial, $c_{kj} \in \bb C$ with $c_{kn(k)}\neq 0$, and where $\alpha_1,\dots,\alpha_m$
are the poles of $r(z)$ that are contained in $\bb C$. Clearly, $r(z) \in \bb C_{\rho(A)}(z)$
if and only if both $\alpha_1,\dots,\alpha_m\in \rho(A)$ and, in the case $\deg p > 0$, $\infty \in \rho(A)$.
Therefore, $r(A)$ as defined below is a well-defined bounded operator. 

\begin{definition}\thlab{ratfufunkdef}
    For $r\in \bb C_{\rho(A)}(z)$ given in the form \eqref{fracdec}, we set
    \[
	r(A):= p(A) + \sum_{k=1}^m \sum_{j=1}^{n(k)} c_{kj} (A - \alpha_k)^{-j} \ \in B(\mc K) \,.\vspace*{-5mm}
    \]
\end{definition}

Using the resolvent equality and in case $\lambda, \infty\in \rho(A)$ also 
$(A-\lambda)^{-1}A = I + \lambda (A-\lambda)^{-1}$, it is straight forward to verify

\begin{theorem}\thlab{ratcalcul}
    For a linear relation $A$ on Banach space $\mc K$ with a non-empty resolvent set,
    the mapping $r \mapsto r(A)$ constitutes an algebra homomorphism from $\bb C_{\rho(A)}(z)$ into $B(\mc K)$, the space
    of all bounded, everywhere defined linear operators on $\mc K$.
		
    If $\mc K$ is a Krein space and $.^+$ denotes the Krein space adjoint as defined in \eqref{kadjde}, 
    then we additionally have $r(A)^+=r^\#(A^+)$, where 
    $r^\#(z)=\overline{r(\bar z)}$.
\end{theorem}

Our functional calculus is compatible with M\"obius type transformations.
In fact, if $M\in \bb C^{2\times 2}$ is regular and if the pole of $\phi_M$ belongs to $\rho(A)$, 
then it easily follows from \eqref{bspfmoeb}, that $\tau_M(A)=\phi_M(A)$.
From this, we can derive the following result on compositions with M\"obius transformations.

\begin{lemma}\thlab{compromise}
    Let $A$ be a linear relation on the Banach space $\mc K$ with a non-empty resolvent set, and
    let $M\in \bb C^{2\times 2}$ be regular. Then $r\mapsto r\circ \phi_M$ constitutes
    an algebra isomorphism from $\bb C_{\rho(\tau_M(A))}(z)$ onto $\bb C_{\rho(A)}(z)$ such that
    \[
	(r\circ \phi_M) (A) = r(\tau_M(A)) \ \ \text{ for all } \ \ r \in \bb C_{\rho(\tau_M(A))}(z) \,.
    \]
\end{lemma}

Finally, a spectral mapping result holds true for our rational functional calculus. For this, recall
($\lambda\in \bb C$)
\[
    \ker (A-\lambda) = \{ x : (x;\lambda x) \in A\} \ \ \text{ and } \ \
    \ker (A-\infty):= \mul A = \{ y : (0;y) \in A\} \,.
\]

\begin{theorem}\thlab{spectransrat}
    For any $r\in \bb C_{\rho(A)}(z)$ we have $\sigma(r(A))=r(\sigma(A))$.	
    Moreover, 
    \[ 
	\ker (A-\lambda) \subseteq \ker (r(A) - r(\lambda)) \ \ \text{ and } \ \
	    \ran (r(A) - r(\lambda)) \subseteq \ran (A-\lambda)
    \]
    hold for all $\lambda \in \bb C\cup\{\infty\}$.
\end{theorem}

\section{Diagonal Transform of Linear Relations}
\label{movlinrel}

In this section we present a method on how to drag a linear relation from a vector space 
to another vector space. After few general definitions and results, we shall focus on
linear relations on Krein spaces. 

To be more precise, we start with an everywhere defined linear operator $T: \mc V \to \mc K$, where 
$\mc V$ and $\mc K$ are vector spaces over the same field $\mathbb R$ or $\mathbb C$. 
If $B$ is a linear relation on $\mc V$, i.e.\ a linear subspace of $\mc V \times \mc V$, then, clearly,
\[
    (T\times T)(B) = \{(Tu;Tv) : (u;v) \in B \} 
\]
is a linear relation on $\mc K$. Similarly, given a linear relation $A$ on $\mc K$, 
i.e.\ a linear subspace of $\mc K \times \mc K$, 
\[
    (T\times T)^{-1}(A) = \{(u;v) \in \mc V \times \mc V : (Tu;Tv) \in A \} 
\]
is a linear relation on $\mc V$. 
If $\mc K$ and $\mc V$ are normed spaces and if $T$ is bounded, then $(T\times T)^{-1}(A)$ is closed for any 
closed $A$.
Besides these trivial facts, it is easy to show that
\[
    (T\times T)(B) = TBT^{-1}, \quad (T\times T)^{-1}(A)=T^{-1}AT \,,
\]
where these products have to be interpreted as relational products. Moreover,
as $\tau_M$ commutes with $T\times T$ for any $M\in\bb C^{2\times 2}$ we have
\begin{equation}\label{zzue4}
\begin{aligned}
 \tau_M \big((T\times T)(B)\big) &=(T\times T)\big(\tau_M(B)\big), \\
\tau_M \big((T\times T)^{-1}(A)\big) &=(T\times T)^{-1}\big(\tau_M(A)\big) \,,
\end{aligned}
\end{equation}
where, for the second equality, $M$ has to be invertible.

\begin{remark}\thlab{comminter}
    Let  $T: \mc V \to \mc K$ be a linear mapping between vector spaces, and let $A$
    be a linear relation on $\mc K$ and $B$ be a linear relation on $\mc V$.
    We consider the condition $(T\times T)(B) \subseteq A$, which clearly is equivalent
    to $B \subseteq (T\times T)^{-1}(A)$.
    Multiplying $T$ from the left 
    to the inclusion $B \subseteq (T\times T)^{-1}(A)= T^{-1} \ A \ T$ yields $TB \subseteq TT^{-1} \ A \ T \subseteq AT$.
    Conversely, $TB \subseteq AT$ implies $B \subseteq T^{-1}TB \subseteq T^{-1} AT = (T\times T)^{-1}(A)$.

    Thus, the conditions $(T\times T)(B) \subseteq A$, $B \subseteq (T\times T)^{-1}(A)$ and
    the intertwining condition $TB \subseteq AT$ are all equivalent.
    If $A$ and $B$ are everywhere defined operators, these conditions are even equivalent to $TB = AT$.
\end{remark}

Using the previous remark we get the following result on the rational functional calculus. 

\begin{proposition}\thlab{ratcalculcom}
    Let $A$ be a linear relation on $\mc K$ with $\rho(A)\neq \emptyset$.
    If $R \in B(\mc K)$ satisfies $(R\times R)(A) \subseteq A$, then $R$ commutes with 
    $r(A)$ for all $r \in \bb C_{\rho(A)}(z)$.
\end{proposition}
\begin{proof}
    By \eqref{zzue4} $(R\times R)(A) \subseteq A$ yields $(R\times R)(\tau_M(A)) \subseteq 
    \tau_M(A)$ for all $M\in\bb C^{2\times 2}$. Consequently, according to \thref{comminter} 
    we have $R\tau_M(A) = \tau_M(A) R$ for all $M\in\bb C^{2\times 2}$ such that $\tau_M(A)$ is a 
    bounded operator. Using \eqref{bspfmoeb}, in particular we have 
    \[ R(A-\lambda)^{-1} = (A-\lambda)^{-1} R \ \ \text{ for all } \ \ \lambda\in\rho(A)\,. \]
    By the definition of $r(A)$, $R$ then commutes with $r(A)$ for all $r \in \bb C_{\rho(A)}(z)$.
\end{proof}

Also based on \thref{comminter}, the subsequent result will prove to be useful. 

\begin{lemma}\thlab{2xanwe}
If $T,S: \mc K \to \mc K$ are linear, everywhere defined operators such that $TS = ST$, then 
\[
    (T \times T)^{-1}(S) = S \boxplus (\ker T \times\ker T) \,,
\]
where $\boxplus$ denotes the usual vector space sum of two subspaces of $\mc K \times \mc K$.
\end{lemma}
\begin{proof}
$TS = ST$ can be expressed as $(T\times T)(S)\subseteq S$ or as $S\subseteq (T\times T)^{-1}(S)$.
Together with $(\ker T \times\ker T) \subseteq (T\times T)^{-1}(S)$, we obtain
$S \boxplus (\ker T \times\ker T) \subseteq (T \times T)^{-1}(S)$.

If $(x;y) \in (T \times T)^{-1}(S)$, then $(x;Sx) \in S \subseteq (T\times T)^{-1}(S)$. Hence, 
$(x;y)-(x;Sx)=(0;y-Sx)\in (T\times T)^{-1}(S)$, and in turn $(0;T(y-Sx))\in S$.
From $\mul S=\{0\}$ we conclude $y-Sx\in \ker T$. 
Thus, $(x;y) = (x;Sx) + (0;y-Sx) \in S \boxplus (\ker T \times\ker T)$.
\end{proof}

\begin{lemma}\thlab{fhwr5}
    Let  $T: \mc V \to \mc K$ be a linear mapping between vector spaces, and let 
    $A_1,A_2$ be a linear relation on $\mc K$ and $\mu\in\bb C\setminus \{0\}$. Then we have
    $(T\times T)^{-1}(\mu A) = \mu \ (T\times T)^{-1}(A)$ and
    \[
	  (T\times T)^{-1}(A_1 + A_2) \supseteq (T\times T)^{-1}(A_1) + (T\times T)^{-1}(A_2) \,, 
    \]
    \[
	  (T\times T)^{-1}(A_1 \ A_2) \supseteq (T\times T)^{-1}(A_1) \ (T\times T)^{-1}(A_2) \,.
    \]
\end{lemma}
\begin{proof}
    The first relation is easy to check.     
    For the second let $(x;y)$ belong to
    $(T\times T)^{-1}(A_1) + (T\times T)^{-1}(A_2)$. This means that $(Tx;Tu)\in A_1$
    and $(Tx;Tv)\in A_2$ for some $u,v\in\mc V$ with $u+v=y$. Hence, $Tu+Tv=Ty$, and
    in turn $(Tx;Ty) \in A_1+A_2$ which yields $(x;y)\in (T\times T)^{-1}(A_1 + A_2)$.
    The final relation is a consequence of $TT^{-1} \subseteq I$:
    \begin{align*}
	(T\times T)^{-1}(A_1) \ (T\times T)^{-1}(A_2) & = (T^{-1}A_1T) \ (T^{-1}A_2T) \\ & \subseteq T^{-1}A_1 A_2 T =
	(T\times T)^{-1}(A_1 \ A_2)
	\,.
    \end{align*}
\end{proof}

\noindent
Recall that the point spectrum of a linear relation $A$ is 
defined by 
\[
  \sigma_p(A)=\{\lambda\in\bb C\cup\{\infty\}: \ker(A-\lambda)\neq \{0\}\} \,.
\]

\begin{lemma}\thlab{cortranpo}
    For a linear relation $A$ on $\mc K$ we have
    \[
	\ker ((T\times T)^{-1}(A) - \lambda) = T^{-1} \ker(A - \lambda) \ \ \text{ for all } \ \ \lambda \in\bb C\cup\{\infty\} 
	\,.
    \]
    In particular, $\sigma_p((T\times T)^{-1}(A)) \subseteq \sigma_p(A)$
    if $T: \mc V \to \mc K$ is injective.
\end{lemma}
\begin{proof}
   First note that 
   \[
	y\in \mul (T\times T)^{-1}(A) \Leftrightarrow (0;Ty) \in A \Leftrightarrow y\in T^{-1}(\mul A) \,.
   \]
   Hence $\ker ((T\times T)^{-1}(A) - \lambda) = T^{-1} \ker(A - \lambda)$
   for $\lambda=\infty$. For the general case we set 
    $M=\bigl(\begin{smallmatrix} 0 & 1 \\ 1 & -\lambda \end{smallmatrix}\bigr)$ and get
\begin{align*}
	\ker ((T\times T)^{-1}(A) - \lambda) & = \mul \tau_M \big( (T\times T)^{-1}(A) \big) =
	\mul (T\times T)^{-1} \tau_M(A) \\ & = T^{-1} \mul \tau_M(A) = T^{-1} \ker(A - \lambda) \,.
\end{align*}
    For injective $T$, $T^{-1} \ker(A - \lambda)\neq \{0\}$ implies $\ker(A - \lambda)\neq \{0\}$.
    Therefore, $\sigma_p((T\times T)^{-1}(A)) \subseteq \sigma_p(A)$.
\end{proof}

In the following result we study the connection between $\dom (T\times T)^{-1}(A)$ and $\dom A$. 

\begin{lemma}\thlab{domrangevinv}
    For a linear relation $A$ on $\mc K$ we have
    \[
	\dom (T\times T)^{-1}(A) \subseteq T^{-1}(\dom A) \,.
    \]
    If $\ran (A - \mu) \subseteq \ran T$ for some $\mu\in \bb C$, then we even have
    \[
	\dom (T\times T)^{-1}(A) = T^{-1}(\dom A) \,. 
    \]
\end{lemma}
\begin{proof}
    It is straight forward to show $\dom (T\times T)^{-1}(A) \subseteq T^{-1}(\dom A)$. 
    
    If $\ran (A-\mu) \subseteq \ran T$ and if $Tx \in \dom A=\dom (A-\mu)$, 
    then $(Tx;v) \in A-\mu$ for some $v\in \ran (A-\mu) \subseteq \ran T$. 
    Hence, $v=Ty$ for some $y\in \mc V$, and in turn,
    $(x;y+\mu x) \in (T\times T)^{-1}(A)$. Thus, $x \in \dom (T\times T)^{-1}(A)$.
%
%
%
\end{proof}

\section{Diagonal Transform on Krein Spaces}
\label{movlinrelrein}

In this section we consider two Krein spaces $(\mc V,[.,.])$ and $(\mc K,[.,.])$ which are linked by
a bounded linear mapping $T: \mc V \to \mc K$.  
Mostly, $(\mc V,[.,.])$ will be a Hilbert space.

In the following, $T^+: \mc K \to \mc V$ denotes the Krein space adjoint operator of $T$.
For a linear relation $A$ on $\mc K$ the adjoint linear relation is defined as
\begin{equation}\label{kadjde}
    A^+ := \{(x;y) \in \mc K\times\mc K : [x,v]=[y,u] \ \text{ for all } \, (u;v) \in A\} \,.
\end{equation}
Obviously, this definition depends on the chosen inner products. 
Note that $A^+=\tau_{M}(A^{[\bot]})=\tau_{M}(A)^{[\bot]}$, where 
$M=\bigl(\begin{smallmatrix} 0 & -1 \\ 1 & 0 \end{smallmatrix}\bigr)$ and
the orthogonal complement is taken in the Krein space $(\mc K\times\mc K, [.,.]_{\mc K\times\mc K})$, where
$[(x;y),(u;v)]_{\mc K\times\mc K} = [x,u] + [y,v]$.
Analogously, the adjoints of relations on $\mc V$ are defined.

\begin{lemma}\thlab{adjungeig}
    For linear relation $A$ on $\mc K$, we have
    \[
	\big((T^+ \times T^+)(A)\big)^+ = (T\times T)^{-1}(A^+) \,.
    \]
    In particular, $\left((T\times T)^{-1}(A^+)\right)^+$ is the closure of $(T^+ \times T^+)(A)$.
\end{lemma}
\begin{proof}
    For a continuous linear $R: \mc K_1 \to \mc K_2$ between Krein spaces and $L\subseteq \mc K_2$ it is easy to check that 
    $R^+(L)^{[\bot]}$ coincides with the inverse image $R^{-1}(L^{[\bot]})$ of $L^{[\bot]}$ under $R$. 
    Applying this to $T\times T$ and $A$ yields
    \[
	\left((T^+ \times T^+)(A)\right)^{[\bot]} = (T\times T)^{-1}(A^{[\bot]}) \,,
    \]
    when we equip $\mc V \times \mc V$ and $\mc K \times \mc K$ with the respective sum scalar product.
    From $A^+ = \tau_{\bigl(\begin{smallmatrix} 0 & -1 \\ 1 & 0 \end{smallmatrix}\bigr)}(A^{[\bot]})=
    \tau_{\bigl(\begin{smallmatrix} 0 & -1 \\ 1 & 0 \end{smallmatrix}\bigr)}(A)^{[\bot]}$ 
    and \eqref{zzue4} we obtain
\begin{align*}
	\left((T^+ \times T^+)(A)\right)^+ & = \tau_{\bigl(\begin{smallmatrix} 0 & -1 \\ 
	      1 & 0 \end{smallmatrix}\bigr)}\left(\left((T^+ \times T^+)(A)\right)^{[\bot]}\right) \\ & = 
	\tau_{\bigl(\begin{smallmatrix} 0 & -1 \\ 
	1 & 0 \end{smallmatrix}\bigr)}\left((T\times T)^{-1}(A^{[\bot]})\right) = (T\times T)^{-1}(A^+) \,.
\end{align*}
    Taking adjoints shows that $(T\times T)^{-1}(A^+)^+$ is the closure of $(T^+ \times T^+)(A)$.
\end{proof}

\begin{proposition}\thlab{kommmitransf}
    Let $T: \mc V \to \mc K$ be a bounded and linear mapping between Krein spaces $\mc V$ and $\mc K$.
    If $A$ is a closed linear relation on $\mc K$, which satisfies 
    \begin{equation} \label{invarvss}
	(TT^+ \times TT^+)(A^+) \subseteq A  \,, 
    \end{equation}
    then the closure $(T\times T)^{-1}(A)^+$ of $(T^+ \times T^+)(A^+)$ is a 
    symmetric linear relation on $\mc V$.

    In the special case that $T$ is injective, that $(\mc V,[.,.])$ is a Hilbert space and that $\bb C\setminus \sigma_p(A)$
    contains points from $\bb C^+$ and from $\bb C^-$,
    the relation $(T\times T)^{-1}(A)$ is self-adjoint.
\end{proposition}
\begin{proof}
    The assumption
    $(T\times T) \ (T^+ \times T^+)(A^+) = (TT^+ \times TT^+)(A^+) \subseteq A$ implies 
    $(T^+ \times T^+)(A^+) \subseteq (T\times T)^{-1}(A)$. Thus, also the closure
    $(T\times T)^{-1}(A)^+$ of 
    $(T^+ \times T^+)(A^+)$ -- see \thref{adjungeig} -- 
    is contained in the closed $(T\times T)^{-1}(A)$.
    Hence $(T\times T)^{-1}(A)^+$ is symmetric.
    
    If $\mc V$ is a Hilbert space, then $(T\times T)^{-1}(A)^+$ not being a self-adjoint relation on $\mc A$ 
    implies that its defect indices are not both equal to zero. This means 
    $\ker((T\times T)^{-1}(A) - \lambda) \neq \{0\}$
    for all $\lambda\in\bb C^+$ or for all $\lambda\in\bb C^-$. 
    Hence the point spectrum of its adjoint
    $(T\times T)^{-1}(A)$ contains all points from the upper halfplane or all points from the lower halfplane.
    
    Due to \thref{cortranpo} we have $\sigma_p((T\times T)^{-1}(A)) \subseteq \sigma_p(A)$.
    Hence, $(T\times T)^{-1}(A)^+$ must be self-adjoint if $\bb C\setminus \sigma_p(S)$
    contains points from $\bb C^+$ and from $\bb C^-$.
\end{proof}

\begin{remark}\thlab{pezcxhpre}
    Note that the condition $(TT^+ \times TT^+)(A^+) \subseteq A$ clearly holds true, if A is closed,
    $A^+$ is symmetric and $(TT^+ \times TT^+)(A^+) \subseteq A^+$.
\end{remark}

\begin{remark}\thlab{pezcxh}
    With the notation and assumptions from \thref{kommmitransf}, we additionally suppose that 
    $T$ is injective and that $\mc V$ is a Hilbert space.
    By \thref{kommmitransf}, $(T\times T)^{-1}(A)^+$ -- subsequently we shall write $(T\times T)^{-1}(A)^*$ 
    since the adjoint is taken in a Hilbert space -- is symmetric. We can formulate a somewhat more general statement.
    In fact, employing \thref{cortranpo} we get
    \[
	\dim \ker((T\times T)^{-1}(A) - \lambda) \leq \dim \ker(A - \lambda) \ \ \text{ for all } 
		\ \ \lambda\in \bb C\cup\{\infty\} \,.
    \]
    Hence the defect indices $n_\pm$ of the symmetry $(T\times T)^{-1}(A)^*$ can be estimated
    from above by $\min \{ \dim \ker(A - \lambda) : \lambda \in \bb C^\pm \}$.
    $\bb C\setminus \sigma_p(A)$ containing points from $\bb C^+$ and from $\bb C^-$ then yields
    $n_\pm = 0$, and we again obtain the result from above.
\end{remark}

The following assertion is a consequence of Loewner's Theorem. In order to be more self-contained we 
give a direct verification using spectral calculus for self-adjoint operators on Hilbert spaces. 

\begin{lemma}\thlab{wuzuab}
    Let $(\mc H,(.,.))$ be a Hilbert space and let $A,C\in B(\mc H)$ such that $C$ and $AC$ are self-adjoint and
    such that $C\geq 0$. Then we have $|(AC x,x)| \leq \|A\| \ (Cx,x)$ for all $x \in\mc H$.
\end{lemma}
\begin{proof}
    Using the functional calculus for the self-adjoint operator $C$ we see that $C+\epsilon$
    is boundedly invertible for any $\epsilon>0$, and $C(C+\epsilon)^{-1}$ has norm 
    $\sup_{t\in \sigma(C)} \frac{t}{t+\epsilon} = \frac{\|C\|}{\|C\|+\epsilon}$. 
    
    Since for the spectral radius we have $\spr(FG)=\spr(GF)$ for all bounded operators $F,G$,
    \[
	\spr((C+\epsilon)^{-\frac{1}{2}}AC(C+\epsilon)^{-\frac{1}{2}}) =
	\spr(AC (C+\epsilon)^{-1}) \leq \|A\| \ \frac{\|C\|}{\|C\|+\epsilon} \,.
    \]
    For self-adjoint operators spectral radius and norm coincide. Hence, due to the Cauchy-Schwarz inequality,
\begin{align*}
	|(AC x,x)| & = |((C+\epsilon)^{-\frac{1}{2}}AC(C+\epsilon)^{-\frac{1}{2}}
		    \ (C+\epsilon)^{\frac{1}{2}}x,(C+\epsilon)^{\frac{1}{2}}x)| \\ & \leq
		    \|(C+\epsilon)^{-\frac{1}{2}}AC(C+\epsilon)^{-\frac{1}{2}}\| \ \|(C+\epsilon)^{\frac{1}{2}}x \|^2
		 \\ & \leq \|A\| \ \frac{\|C\|}{\|C\|+\epsilon} ((C+\epsilon) x,x) \,.
\end{align*}
    The desired inequality follows for $\epsilon\searrow 0$.
\end{proof}

The following result can easily be derived from the spectral calculus for self-adjoint operators on Hilbert spaces. 
We omit the details.

\begin{lemma}\thlab{selbadjbeschr}
    Let $(\mc H,(.,.))$ be a Hilbert space, $c\in [0,+\infty)$ and let $B$ be a self-adjoint operator.
    If $|(Bx,x)| \leq c (x,x)$ for $x\in \dom B$, then $B$ is bounded with $\|B\| \leq c$.
\end{lemma}

The ideas in the subsequent lemma are taken from \cite{dritschelrovnyak1996}.

\begin{lemma}\thlab{contranspo}
    With the notation and assumptions from \thref{kommmitransf} additionally suppose that 
    $T$ is injective, that $\mc V$ is a Hilbert space and that
    $A: \mc K \to \mc K$ is bounded. Then $(T\times T)^{-1}(A)$ is a bounded
    linear and self-adjoint operator on $\mc V$ with 
    \begin{equation}\label{contranspoeq}
	\| \ (T\times T)^{-1}(A) \ \| \leq \|A\| \,.
    \end{equation}
    Here $\|.\|$ on the right is the operator norm with respect to any Hilbert space scalar product $(.,.)$
    on $\mc K$ compatible with $[.,.]$, i.e.\ $[.,.]=(G.,.)$ for some 
    $(.,.)$-self-adjoint, bounded and boundedly invertible Gram operator
    $G: \mc K \to \mc K$.
\end{lemma}
\begin{proof}
    $\sigma(A) \subseteq K_{\|A\|}(0)$ yields $\bb C \setminus K_{\|A\|}(0) 
    \subseteq (\bb C\cup\{\infty\}) \setminus \sigma_p(A)$.
    In particular, $\bb C\setminus \sigma_p(A)$ contains points from $\bb C^+$ and from $\bb C^-$.
    
    By \thref{kommmitransf} the relation $(T\times T)^{-1}(A)$ is self-adjoint and coincides with the closure 
    of $(T^+ \times T^+)(A^+)$; see \thref{pezcxhpre}. 
    According to \thref{cortranpo} (applied with $\lambda=\infty$) 
    we also know that $(T\times T)^{-1}(A)$ is an operator.    

    Due to \thref{comminter} by our assumption \eqref{invarvss} we have 
    $TT^+ A^+ = A TT^+$. Since the adjoints with respect to the respective scalar products on $\mc K$ are 
    related by $T^+ = T^{*} G$ and $A^+ = G^{-1}A^{*} G$, we have 
    $TT^{*}A^{*} G = TT^+A^{+} = A TT^+ = A TT^{*} G$.
    Consequently, $(A TT^{*})^{*} = TT^{*}A^{*} = A TT^{*}$
    is self-adjoint on the Hilbert space $(\mc K,(.,.))$. 
    
    For $(x;y) \in (T^+ \times T^+)(A^+) \subseteq (T\times T)^{-1}(A)$ we have $x=T^+ u$
    for some $u\in \mc K$. We conclude that $(TT^+u;Ty) \in A$ or $A(TT^+u)=Ty$, and hence
    \[
	|[y,x]| = |[y,T^+u]| = |[Ty,u]| = |[ATT^+ u,u]| = |(ATT^{*} G u, Gu)| \,.
    \]
    By \thref{wuzuab} this expressions is less or equal to  
    \[
	\|A\| \ (TT^{*} G u, Gu) = \|A\| \ [TT^+ u,u] = \|A\| \ [x,x] \,.
    \]
    $(T^+ \times T^+)(A^+)$ being dense in $(T\times T)^{-1}(A)$ implies
    $|[y,x]| \leq \|A\| \ [x,x]$ for all $(x;y)\in (T\times T)^{-1}(A)$. Therefore,
    according to \thref{selbadjbeschr}, $(T\times T)^{-1}(A)$ is a bounded and self-adjoint operator
    with norm less or equal to $\|A\|$.
\end{proof}

\begin{theorem}\thlab{thetadefeig}
    Let $T: \mc V \to \mc K$ be a bounded and injective linear mapping from the Hilbert space $(\mc V,[.,.])$ 
    into the Krein space $(\mc K,[.,.])$. Then
    \[
	\Theta : C \mapsto (T\times T)^{-1}(C)
    \]
    constitutes a bounded $*$-algebra homomorphism from 
    $(TT^+)' \ (\subseteq B(\mc K))$ into
    $(T^+T)' \ (\subseteq B(\mc V))$, where $(TT^+)'$ denotes commutant of 
    $TT^+$ in $B(\mc K)$ and $(T^+T)'$ denotes commutant of 
    $T^+T$ in $B(\mc V)$. 
    Hereby, $\Theta(I) = I$, $\Theta(TT^+) = T^+T$, and
    \[
	\ker \Theta=\{C\in (TT^+)': \ran C \subseteq \ker T^+\} \,.
    \]
    Moreover, $(T^+ \times T^+)(C)$ is densely contained in $\Theta(C)$ for all $C\in (TT^+)'$, and
    we have $T^+ C =  \Theta(C) T^+$.
\end{theorem}
\begin{proof}
    It is easy to check that $(TT^+)' \subseteq B(\mc K)$ and $(T^+T)' \subseteq B(\mc V)$ are
    closed $*$-subalgebras when provided with $.^+$ and $.^{*}$, respectively.
    By \thref{comminter} we have
    $(TT^+ \times TT^+)(C^+)=(TT^+ \times TT^+)(C) \subseteq C$ for any self-adjoint $C\in (TT^+)'$. 
    Thus, we can apply \thref{contranspo} and see that $(T\times T)^{-1}(C)$ is a bounded self-adjoint linear mapping
    on $\mc V$ containing $(T^+ \times T^+)(C)$ densely. Due to
    \[
	(T^+T \times T^+T) \ (T\times T)^{-1}(C) \subseteq (T^+ \times T^+)(C) \subseteq (T\times T)^{-1}(C) 
    \]
    and \thref{comminter} we have $(T\times T)^{-1}(C) \in (T^+T)'$. 
    Clearly, 
    $(T\times T)^{-1}(I) = T^{-1}T=I$ and $(T\times T)^{-1}(TT^+) = T^{-1} TT^+ T = T^+ T$.
    
    For a not necessarily self-adjoint $C\in (TT^+)'$ we also have $C^+\in (TT^+)'$.
    Hence, $C=\Re C + i \Im C$ and $C^+=\Re C - i \Im C$ with  
    \[
	\Re C=\frac{C+C^+}{2}, \ \ \Im C= \frac{C-C^+}{2i} \in (TT^+)' \,.
    \]
    Consequently, $(T\times T)^{-1}(\Re C)$ and $(T\times T)^{-1}(\Im C)$ are self-adjoint
    elements from $(T^+T)'$. Moreover, by \thref{fhwr5}
\[
	(T\times T)^{-1}(\Re C + i\Im C) \supseteq 
				(T\times T)^{-1}(\Re C) + i(T\times T)^{-1}(\Im C) \,,
\]
\[
	(T\times T)^{-1}(\Re C - i\Im C) \supseteq  
				  (T\times T)^{-1}(\Re C) - i(T\times T)^{-1}(\Im C) \,,
\]
    where the right hand sides have domain $\mc V$ and the left hand sides are operators; see \thref{contranspo}.
    Consequently, equalities prevail, and we obtain $(T\times T)^{-1}(C)\in (T^+T)'$ and $(T\times T)^{-1}(C^+) 
    = (T\times T)^{-1}(C)^{*}$. 
    Therefore, $\Theta: (TT^+)' \to (T^+T)'$ is well-defined and satisfies $\Theta(C^+)=\Theta(C)^{*}$.
    Using \thref{fhwr5} two more times shows that $\Theta$ is linear and multiplicative.
    Employing \eqref{contranspoeq} we get ($G$ and $(.,.)$ are as in \thref{contranspo})
    \begin{multline*}
	\|\Theta(C)\|^2 = \sup_{x\in\mc V, [x,x]=1} [\Theta(C)x,\Theta(C)x] = 
	\sup_{x\in\mc V, [x,x]=1} [\Theta(C^+C) x,x] \leq \\ \|\Theta(C^+C)\| \leq \|C^+C\| = 
	\|G^{-1}C^{*}G C\| \leq \|G^{-1}\| \, \|G\| \cdot \|C\|^2 \,, 
    \end{multline*}
    and conclude that $\Theta$ is bounded. From \thref{adjungeig} we infer
    \[
	\left((T^+ \times T^+)(C)\right)^{*} = (T\times T)^{-1}(C^{+}) = (T\times T)^{-1}(C)^{*} 
    \]
    showing that $(T^+ \times T^+)(C)$ is densely contained in $(T\times T)^{-1}(C)$. In particular,
    $(T\times T)^{-1}(C)=\Theta(C)=0$ is equivalent to the fact that $(a;b) \in (T^+ \times T^+)(C)$
    always implies $b=0$, i.e., $T^+ y = 0$ for all $(x;y) \in C$. This just means that 
    $\ran C$ is contained in $\ker T^+$.
    
    Finally, we have $[TT^+ C u,v] = [T^+ C u, T^+ v] = [\Theta(C) T^+ u, T^+ v]$
    for any $u,v\in \mc K$ because of $(T^+ u; T^+ C u) \in \Theta(C)$.
    From this equality we obtain $TT^+ C = T\Theta(C) T^+$ which by $T$'s injectivity implies
    $T^+ C =  \Theta(C) T^+$.
\end{proof}

\begin{remark}\thlab{thetaremretour}
    For $C\in (TT^+)'$ we can apply \thref{2xanwe}, and obtain
    \[
	(T^+\times T^+)^{-1} \Theta(C) =
	(TT^+\times TT^+)^{-1}(C) = C \boxplus (\ker TT^+ \times\ker TT^+) \,,
    \]
    where $\ker TT^+ = \ker T^+$ by $T$'s injectivity.
\end{remark}

For linear relations with non-empty resolvent sets, we can apply the previous result to
a M\"obius type transformation of the given linear relation.

\begin{corollary}\thlab{umktheta}
    Let $T: \mc V \to \mc K$ be a bounded and injective linear mapping from the Hilbert space $(\mc V,[.,.])$ 
    into the Krein space $(\mc K,[.,.])$. Let $C$ be a linear relation on $\mc K$ with $\rho(C)\neq \emptyset$
    and $(TT^+ \times TT^+)(C)\subseteq C$.
		
    Then, the linear relation $\Theta(C):=(T\times T)^{-1}(C)$ on $\mc V$ densely contains
    $(T^+ \times T^+)(C)$ and satisfies $(T^+T \times T^+T)(\Theta(C)) \subseteq \Theta(C)$.
    Moreover, $C^+$ also satisfies $(TT^+ \times TT^+)(C^+)\subseteq C^+$ 
    and $\Theta(C^+) = \Theta(C)^*$.
    Finally, $\rho(\Theta(C)) \supseteq \rho(C)$, and
    $\Theta(r(C)) = r(\Theta(C))$ holds true for all $r\in \bb C_{\rho(C)}(z)$.
\end{corollary}
\begin{proof}
    We apply $\tau_M$ with $M=\bigl(\begin{smallmatrix} 0 & 1 \\ 1 & -\lambda \end{smallmatrix}\bigr)$ 
    to $(TT^+ \times TT^+)(C)\subseteq C$ and obtain  
    $(TT^+ \times TT^+)((C-\lambda)^{-1})\subseteq (C-\lambda)^{-1}$ for any $\lambda \in \rho(C)$; see
    \eqref{zzue4}.
    
    Since then $(C-\lambda)^{-1}$ commutes with $TT^+$ (see \thref{comminter}), by \thref{thetadefeig} 
    $(T\times T)^{-1}((C-\lambda)^{-1})$ is a bounded operator commuting with $T^+T$, i.e.
    \begin{equation}\label{zklf5}
      (T^+T \times T^+T)(T\times T)^{-1}\big((C-\lambda)^{-1}\big) 
		    \subseteq (T\times T)^{-1}\big((C-\lambda)^{-1}\big) \,.
    \end{equation}
    Moreover, $(T^+ \times T^+)((C-\lambda)^{-1})$ is densely contained in $(T\times T)^{-1}\big((C-\lambda)^{-1}\big)$.
    
    $(T\times T)^{-1}((C-\lambda)^{-1}) = (\Theta(C) - \lambda)^{-1}$ gives
    $\lambda \in \rho(\Theta(C))$. 
    Applying $\tau_{M^{-1}}$ to \eqref{zklf5} yields $(T^+T \times T^+T)(\Theta(C)) \subseteq \Theta(C)$. Since 
    $\tau_{M^{-1}}$ is bi-continuous,  
    $\tau_{M^{-1}}(T^+ \times T^+)((C-\lambda)^{-1}) = (T^+ \times T^+)(C)$
    is densely contained in $\Theta(C)=\tau_{M^{-1}}(T\times T)^{-1}\big((C-\lambda)^{-1}\big)$.
    
    Concerning $C^+$ note that by \thref{comminter},
    $(TT^+ \times TT^+)(C)\subseteq C$ is equivalent to $T T^+ C \subseteq C T T^+$. 
    Taking adjoints gives $T T^+ C^+ \subseteq  C^+T T^+$, which in turn is equivalent to
    $(TT^+ \times TT^+)(C^+)\subseteq C^+$. 
    With the help of \thref{adjungeig} and the fact that $(T^+ \times T^+)(C)$ is dense in $\Theta(C)$, we get
\begin{align*}
	\Theta(C^+) & = (T\times T)^{-1} (C^+) = (T^+ \times T^+)(C)^* = \Cl((T^+ \times T^+)(C))^* \\ & =
	(T\times T)^{-1}(C)^* = \Theta(C)^*.
\end{align*}
    Finally, $\Theta(r(C)) = r(\Theta(C))$ for $r\in \bb C_{\rho(C)}(z)$ follows in a straight forward manner from
    $(T\times T)^{-1}((C-\lambda)^{-1}) = (\Theta(C) - \lambda)^{-1}$.
\end{proof}

Via $\Theta$ we can drag certain linear relation on $\mc K$ to linear relations on $\mc V$.
We now present a way how to drag at least operators into the other direction.

\begin{lemma}\thlab{Xidefeig}
    Let $T: \mc V \to \mc K$ be a bounded and injective linear mapping from the Hilbert space $(\mc V,[.,.])$ 
    into the Krein space $(\mc K,[.,.])$. Then
    \[
	\Xi : D \mapsto T D T^+
    \]
    maps 
    $(T^+T)' \ (\subseteq B(\mc V))$ linearly and boundedly into $(TT^+)' \ ( \subseteq B(\mc K))$ and
    satisfies ($C, \in (TT^+)', \ D,D_1,D_2\in (T^+T)'$)
    \begin{align*}
	& \Xi(D^{*}) =  \Xi(D)^+, \ \ \
	\Xi(D \ \Theta(C)) = \Xi(D) C, \ \ \ \Xi(\Theta(C) \ D) = C \Xi(D), \\  
	& \Xi(D_1D_2 \ T^+T) = \Xi(D_1) \, \Xi(D_2), \ \ \
	\Xi\circ \Theta (C) = TT^+ \ C = C \ TT^+ \,. 
    \end{align*}
    Moreover, $\Xi(D)$ commutes with all operators from 
    $(TT^+)'$ if $D$ commutes with all operators from $(T^+T)'$, i.e.\ $\Xi((T^+T)'') \subseteq (TT^+)''$. 
\end{lemma}
\begin{proof}
    $\Xi : D \mapsto T D T^+$ is clearly linear and it is bounded by 
    $\|T\| \ \|T^+\|$. Obviously, it satisfies 
    $\Xi(D)^+ = \Xi(D^{*})$. Its injectivity follows from $T$'s injectivity and from 
    $\Cl(\ran T^+) = \ker T^{\bot} = \mc V$. For $D\in (T^+T)'$ we have 
    \[
	\Xi(D) \ TT^+ = T D T^+ \ TT^+ = T \ T^+ T D T^+ = TT^+ \ \Xi(D) \,,
    \]
    i.e.\ $\Xi(D)\in (TT^+)'$. For $C\in (TT^+)', D \in (T^+T)'$ due to 
    $T^+ C =  \Theta(C) T^+$ we have  
    $\Xi(D\Theta(C)) = T D\Theta(C) T^+ = T D T^+ C = \Xi(D) C$.
    Applying this to $D^*, C^+$ and taking adjoints yields $\Xi(\Theta(C) D) = C \Xi(D)$.
    
    For $D_1,D_2 \in (T^+T)'$ we have
    \[
	\Xi(D_1D_2 \ T^+T) = T D_1 D_2 T^+TT^+ = T D_1 \ T^+T \ D_2 T^+ =\Xi(D_1) \ \Xi(D_2) \,,
    \]
    and $T^+ C =  \Theta(C) T^+$ 
    implies $\Xi\circ \Theta (C) = T \Theta (C) T^+ = TT^+ \ C = C \ TT^+$.
    
    Finally, assume that $D$ commutes with all operators from $(T^+T)'$, and let $C\in (TT^+)'$. Then
    $T^{-1} C T= \Theta(C) \in (T^+T)'$. Hence
    \[
	\Xi(D) C = \Xi(D \ \Theta(C)) = \Xi(\Theta(C) \ D) = C \Xi(D) \,.
    \]
\end{proof}

\section{Definitizable Linear Relations}
\label{deflinrel}

We start the present section with the definition of definitizability as given in \cite{jonas2003}, Section 4.

\begin{definition}\thlab{definitzdef}
    Let $(\mc K,[.,.])$ be a Krein space. A linear relation $A$ on $\mc K$ 
    is called \emph{definitizable} if $\rho(A)\neq \emptyset$ and $[q(A)x,x] \geq 0$ for all $x \in \mc K$ and 
    some rational
    $q\in \bb C_{\rho(A)}(z)$. Any rational $q\in \bb C_{\rho(A)}(z)$ satisfying this condition is called 
    \emph{definitizing rational function} for $A$. 
\end{definition}

\begin{example}\thlab{pontrunit}
  If $(\mc P,[.,.])$ is a Pontryagin space and $U: \mc P \to \mc P$ is a unitary bounded linear operator,
  then $U$ is definitizable. To see this recall for example from \cite{LaSa1990} that $[p(U)x,p(U)x] \geq 0$
  for some polynomial $p\in \bb C[z]$. Thus, we have
  $[p(U)^+p(U)x,x]=[p(U)x,p(U)x] \geq 0$ for all $x\in\mc P$.
  As $p(U)^+p(U) = p^\#(U^+) p(U) = p^\#(U^{-1}) p(U) = q(U)$, where
  $q(z):=p^\#(\frac{1}{z}) p(z)$ is a rational functions with poles at most in
  $\{0,\infty\} \subseteq \rho(U)$. Thus, $q\in \bb C_{\rho(U)}(z)$.
\end{example}
  
\begin{example}\thlab{pontrsa}
    If $A$ is a self-adjoint linear relation on the Pontryagin space $(\mc P,[.,.])$ with $\rho(A)\neq\emptyset$, then
    taking $\mu \in \rho(A) = \overline{\rho(A)}$ with strictly positive imaginary part, we know
    from \cite{dijksmadesnoo1987a} that the Cayley transform $\mc C_\mu(A)$ is unitary. For 
    $M=\smm[ 1 & - \mu \\ 1 & -\bar \mu ]$ we have $\mc C_\mu(A)=\tau_M(A)$ and 
    from \thref{spectrtrans} we obtain that $\phi_M(\mu)=0$ and $\phi_M(\bar \mu)=\infty$ belong to
    $\rho(\mc C_\mu(A))$. As we saw above, $q(\mc C_\mu(A))$ is positive for some
    $q\in \bb C_{\rho(\mc C_\mu(A))}(z)$. Since $q\circ \phi_M(A) = q(\mc C_\mu(A))$ with
    $q\circ \phi_M \in \bb C_{\rho(A)}(z)$, $A$ turns out to be definitizable.  
\end{example}

According to the following lemma, a definitizable linear relation 
gives rise to the situation discussed in the previous sections.

\begin{lemma}\thlab{ht20g95}
    Let $A$ be a definitizable linear relation on a Krein space $\mc K$ with a
    definitizing rational function $q\in \bb C_{\rho(A)}(z)$. Then there exists
    an, up to isomorphisms, unique Hilbert space $\mc V$ and an injective and bounded linear mapping $T: \mc V \to \mc K$
    such that $TT^+ = q(A)$.
\end{lemma}
\begin{proof}
$\langle .,. \rangle:=[ q(A) .,. ]$ defines a positive semidefinite hermitian 
sesquilinear form on $\mc K$. By $(\mc V,\langle .,. \rangle)$ we denote the Hilbert space completion 
of $(\mc K/\mc K^{\langle \circ \rangle}, \langle .,. \rangle)$, and let 
$\iota: \mc K \to \mc V$ be defined by $\iota(x)= x + \mc K^{\langle \circ \rangle}$ for $x\in\mc K$.
Taking any compatible Hilbert space scalar product $(.,.)$ on $\mc K$ and denoting
the corresponding Gram operator by $G\in B(\mc K)$ we have
\[
    \langle \iota x,\iota x \rangle = [ q(A) x,x ] \leq \|G\| \, \|q(A)\| \ (x,x) \ \ \text{ for all } \ \ x \in \mc K \,.
\]
Hence, $\iota$ is bounded, and its adjoint $T:=\iota^+$ is a bounded linear operator
from $\mc V$ into $\mc K$.
By definition, the range of $\iota=T^+$ is dense, and therefore, $T$ is injective.
Moreover, due to
\[
	[TT^+x,y] = \langle T^+x,T^+y \rangle = \langle x,y \rangle = [ q(A) x,y ] \ \ \text{ for all } \ \  x,y \in \mc K \,,
\]
we have $TT^+ = q(A)$. Concerning the uniqueness let $\mc H$ be a Hilbert space and let
$S: \mc H \to \mc K$ be an injective and bounded linear mapping
such that $SS^+ = q(A)$. Then $S^+$ has dense range, and $U:= \{ (S^+x;T^+x) : x \in \mc K\}$
is an isometric linear relation between the Hilbert spaces $\mc H$ and $\mc V$. Therefore,
the closure of $U$, also denoted by $U$, constitutes a unitary mapping satisfying
$T^+=US^+$.
\end{proof}

\begin{remark}\thlab{ht20g95rem}
    Since $\ran T^+$ is dense in $\mc V$, in \thref{ht20g95} we have $\mc V = \{0\}$ if and only if $q(A)=0$
\end{remark}

\begin{remark}\thlab{ht20g95rem2}
Next, let us verify that our definitizing relation $A$ is actually in the domain 
of the mapping $\Theta$ defined in \thref{umktheta}:

For any regular matrix $M\in\bb C^{2\times 2}$ with $\tau_M(A) \in B(\mc K)$, 
we know $\tau_M(A) q(A) = q(A) \tau_M(A)$; see \thref{ratcalcul}. 
By \thref{comminter}, this intertwining condition is equivalent to 
$(TT^+\times TT^+) \tau_M(A) \subseteq \tau_M(A)$. As pointed out in \eqref{zzue4}, 
applying $\tau_{M^{-1}}$ gives $(TT^+\times TT^+) A \subseteq A$.

Thus, we can apply \thref{umktheta} to the linear relation $A$. 
In particular  $\Theta(A)^* = \Theta(A^+) = \Theta(A)$ if $A=A^+$. Also note that
$\Theta(q(A)) = \Theta(TT^+) = T^+T$ due to \thref{thetadefeig}.
\end{remark}

\begin{theorem}\thlab{positopspecpost}
    Let $(\mc K,[.,.])$ be a Krein space and let $A$ be a definitizable linear relation on $\mc K$. 
    If $q$ is a definitizing rational function for $A$ and if $s\in \bb C_{\rho(A)}(z)$, then
    for $p(z):=s(z)q(z) \in \bb C_{\rho(A)}(z)$ we have 
    \[
	\sigma(\Theta(A)) \subseteq \sigma(A) \subseteq p^{-1}\Big( p(\sigma(\Theta(A))) \cup \{0\}\Big) \,,	 
    \]
    where $\sigma(\Theta(A))$ has to be interpreted as $\emptyset$ for $q(A)=0$.
    Here, $\Theta$ is the mapping as in \thref{umktheta} applied to the situation of \thref{ht20g95}.
\end{theorem}
\begin{proof}
    For $q(A)=0$ and, hence $p(A)=0$, the assertion immediately follows from the 
    Spectral Mapping \thref{spectransrat}. 

    Thus, we can assume that $q(A)\neq 0$.
    The inclusion $\sigma(\Theta(A)) \subseteq \sigma(A)$ was shown in \thref{umktheta}.
    For the second inclusion
    assume $\lambda \not \in p^{-1}\big(p( \sigma(\Theta(A))) \cup \{0\}\big)$. Then    
    $p(\lambda) \not =0$ and $p(\lambda) \not\in p\big( \sigma(\Theta(A)) \big) = \sigma\big(\Theta(p(A))\big)$; 
    see \thref{spectransrat} and \thref{umktheta}.
    Hence, for $M=\bigl(\begin{smallmatrix} 1 & 0 \\ 1 & -p(\lambda) \end{smallmatrix}\bigr)$
    \begin{align}\label{hwr567}\nonumber
    \Theta\big(\tau_M(p(A))\big) &
	= \tau_M\big(\Theta(p(A))\big) = I + p(\lambda) \big(\Theta(p(A)) - p(\lambda)\big)^{-1} \\ & = 
	\Theta(p(A)) \big(\Theta(p(A)) - p(\lambda)\big)^{-1} 
    \end{align}
    is an everywhere defined and bounded linear operator on $\mc V$, whose range is contained in $\ran \Theta(p(A))$.
    According to \thref{thetaremretour}, we then have
    \[
	(T^+\times T^+)^{-1} \Theta\big(\tau_M(p(A))\big) = \tau_M(p(A)) \boxplus (\ker TT^+ \times\ker TT^+) \,.
    \]
    From \eqref{hwr567} and $\Theta(p(A))=\Theta(q(A))\Theta(s(A)) = T^+T \, \Theta(s(A))$ we derive
    \[
	\ran \tau_M(\Theta(p(A))) \subseteq \ran \Theta(p(A)) \subseteq \ran T^+T \subseteq \ran T^+ \,.
    \]
    Therefore, by \thref{domrangevinv} 
    \[
	\dom (T^+\times T^+)^{-1} \Theta\big(\tau_M(p(A))\big) = 
				(T^+)^{-1} \dom \tau_M\big(\Theta(p(A))\big) = \mc K \,,
    \]
    and in turn
    \[
	\mc K = \big(\dom \tau_M(p(A)) \boxplus (\ker TT^+ \times\ker TT^+)\big) = 
			\ran (p(A)-p(\lambda)) + \ker TT^+ \,.
    \]
    $\ker TT^+ = \ker q(A) \subseteq \ker p(A) \subseteq \ran (p(A)-p(\lambda))$ for $p(\lambda) \neq 0$ yields
    $\ran (p(A)-p(\lambda))=\mc K$. From \thref{spectransrat} we conclude $\ran (A-\lambda) = \mc K$.
    
    Finally, take $x \in \ker (A-\lambda) \subseteq \ker (p(A)-p(\lambda))$.
    \thref{thetadefeig} gives
    \[
	(T^+ x; T^+ p(\lambda) x) \in (T^+\times T^+) p(A) \subseteq \Theta(p(A)) \,.
    \]
    In particular,
    $T^+x \in \ker(\Theta(p(A))-p(\lambda))= \{0\}$, i.e. $T^+ x = 0$.
    The calculation
    \[
	p(\lambda) x = p(A) x = s(A) q(A) x = s(A) T T^+ x = 0 \,,
    \]
    shows $x=0$. Thus, we conclude $\lambda \in \rho(A)$.
\end{proof}

If for bounded $A$ the assumptions from the previous theorem are satisfied for $p(z)=q(z)=z$, then 
$\Theta(A)=\Theta(p(A))=\Theta(TT^+) = T^+T$ is self-adjoint in the Hilbert space $\mc V$. 
Hence $\sigma(\Theta(A)) \subseteq \bb R$ and we obtain the following well known result as a corollary; 
see \cite{azio1989}.

\begin{corollary}\thlab{positopspec}
    Let $(\mc K,[.,.])$ be a Krein space and assume that $A: \mc K \to \mc K$ is a bounded, 
    linear operator, such that $[Ax,x] \geq 0$ for all $x\in \mc K$, i.e. $A$ is positive. 
    Then $\sigma(A)\subseteq \bb R$. 
\end{corollary}

\begin{corollary}\thlab{defspecpro}
    Let $(\mc K,[.,.])$ be a Krein space,  let $A$ be a definitizable linear relation on $\mc K$ 
    and let $q$ be a definitizing rational function for $A$. 
		
    If $A$ is in addition self-adjoint, then
    \[
		\sigma(A) \subseteq \bb R \cup \{\infty\} \cup \{z \in \bb C : q(z)=0\} \,,
    \]
    where
    $\sigma(A) \cap \{z \in \bb C : q(z)=0\}$ is symmetric with respect to $\bb R$.
    Moreover, 
    \[
      \sigma(\Theta(A)) \subseteq \sigma(A) \subseteq \sigma(\Theta(A)) \cup \{z \in \bb C : q(z)=0\} \,.
    \]
    Here, $\Theta$ is the mapping as in \thref{umktheta} applied to the situation of \thref{ht20g95}.
\end{corollary}
\begin{proof}
    By assumption, $q(A)$ is a positive and bounded operator. \thref{positopspec} therefore gives
    $q(\sigma(A))=\sigma(q(A)) \subseteq \bb R$.
    
    For any $\infty\not= \mu\in \rho(A)$ we also have $\bar \mu\in\rho(A)$. Consider 
    $s_\mu(z):=\frac{1}{z-\mu}, \ s_\mu^\#(z)= \frac{1}{z-\bar\mu} \in \bb C_{\rho(A)}(z)$.
    \thref{ratcalcul} together with $A=A^+$ shows that
    \[
	[ (qs_\mu s_\mu^\#)(A) x,x] = [q(C) \ s_\mu(A) x, s_\mu(A) x] \geq 0\ \ \text{ for all} \ \ x \in \mc K \,.
    \]
    Hence we also have $(qs_\mu s_\mu^\#)(\sigma(A)) = \sigma((qs_\mu s_\mu^\#)(A)) \subseteq \bb R$. 
    
    Assume that $z\in \sigma(A)\setminus \bb R$ and $q(z)\neq 0$.
    From $q(z)\in \bb R$ we conclude
\begin{align*}
	\frac{1}{(z-\mu)(z-\bar \mu)} q(z) - & \overline{\frac{1}{(z-\mu)(z-\bar \mu)} q(z)} = \\ & =  
	\frac{q(z)}{|z-\mu|^2 |z-\bar \mu|^2} \left( {\bar z}^2  - z^2 - 2\Re \mu(\bar z - z)\right) \\ &  =
	\frac{2 \ q(z)}{|z-\mu|^2 |z-\bar \mu|^2} (\bar z - z)(\Re z - \Re \mu) \,.
\end{align*}
    For $\Re \mu \neq \Re z$ this term does not vanish, i.e., $(qs_\mu s_\mu^\#)(z) \not\in \bb R$. 
    Since $\rho(A)$ is open, $\Re \mu \neq \Re z$ can always be achieved by perturbing $\Re \mu$ a little, and 
    we obtain a contradiction to $(qs_\mu s_\mu^\#)(\sigma(A)) \subseteq \bb R$.
    
    For $w\in \sigma(A) \cap \{z \in \bb C : q(z)=0\}$ the self-adjointness of $A$ yields
    $\bar w \in \sigma(A) \subseteq \bb R \cup \{\infty\} \cup \{z \in \bb C : q(z)=0\}$.
    Hence $\bar w \in \sigma(A) \cap \{z \in \bb C : q(z)=0\}$.

    Finally, assume that $\lambda \not\in \sigma(\Theta(A)) \cup \{z \in \bb C : q(z)=0\}$ but 
    $\lambda \in \sigma(A)$. Hence $\lambda \in \bb R\cup \{\infty\}$. Let $U(\lambda)$ be a compact,
    and with respect to $\bb R$ symmetric neighbourhood of $\lambda$ such that
    \[
      U(\lambda)\cap (\sigma(\Theta(A)) \cup \{z \in \bb C : q(z)=0\})=\emptyset \,.
    \]
    Since $q\big(\sigma(\Theta(A))\big)=\sigma\big(q(\Theta(A))\big)$ is bounded in $\bb C$, the same is true for
    \[
    \bigcup_{\mu\in U(\lambda)} (qs_\mu s_\mu^\#)(\sigma(\Theta(A))) \subseteq 
    \frac{1}{\sigma(\Theta(A)) - U(\lambda)}\cdot \frac{1}{\sigma(\Theta(A)) - U(\lambda)} \cdot q(\sigma(\Theta(A))) \,.
    \]
    On the other hand, for any sequence $\mu_n\in U(\lambda)\setminus (\bb R\cup \{\infty\}) \subseteq \rho(A), \ n\in\bb N$,
    with $\lim_{n\to\infty} \mu_n=\lambda$ we have
    \[
	(qs_{\mu_n} s_{\mu_n}^\#)(\sigma(A)) \ni (qs_{\mu_n} s_{\mu_n}^\#)(\lambda) = 
	\frac{q(\lambda)}{|\lambda - \mu_n|^2} \to +\infty \,,
    \]
    which contradicts the boundedness of
\begin{align*}
	\bigcup_{n\in\bb N} (qs_{\mu_n} s_{\mu_n}^\#)(\sigma(A)) & \subseteq 
	\bigcup_{n\in\bb N} (qs_{\mu_n} s_{\mu_n}^\#)\big(\sigma(\Theta(A))\big) \cup \{0\}
	\\ & \subseteq \bigcup_{\mu\in U(\lambda)} (qs_\mu s_\mu^\#)\big(\sigma(\Theta(A))\big) \cup \{0\} \,,
\end{align*}
    where the first inclusion follows from \thref{positopspecpost}.
\end{proof}

\begin{remark}\thlab{definitzreal}
    According to \thref{defspecpro} the zeros of $q$, that lie in $\sigma(A)$, are symmetric with respect to $\bb R$.
    If we consider $s:=q+q^\#$, where again $q^\#(z)=\overline{q(\bar z)}$, then 
    $s^\#=s$ and $s(A)=q(A)+q^\#(A)=q(A)+q(A)^{*} = 2 q(A)$. Hence with $q$
    also $s:=q+q^\#$ is a definitizing rational function. The latter is real, i.e.\ $s^\#=s$.
    
    If $\sigma(A)$ is not finite, then by \thref{defspecpro} also $\sigma(A)\cap \bb R$ is not finite.
    From $q(\sigma(A)\cap \bb R)\subseteq q(\sigma(A)) = \sigma(q(A)) \subseteq \bb R$ 
    we conclude that $q$ and $q^\#$ coincide at least on an infinite subset of $\bb C\cup\{\infty\}$. By holomorphy
    they coincide everywhere, i.e.\ $q=q^\#$ for any definitizing rational function $q$ in case that $\sigma(A)$ is not finite.
\end{remark}

\section{Functional Calculus for Self-adjoint Definitizable Linear Relations}

In the present section, we derive a functional calculus for self-adjoint definitizable linear relations
very similar to the functional calculus for self-adjoint operators on Hilbert spaces, which 
assigns to each bounded and Borel measurable
function $\phi$ on the spectrum the operator $\int \phi \, dE$. 

First, let us recall from \cite{dijksmadesnoo1987a} the following
pretty straight forward generalization of the Spectral Theorem for self-adjoint linear relations on Hilbert spaces.

\begin{theorem}\thlab{specsarelhi}
Let $\mc H$ be a Hilbert space and let $A \subseteq \mc H \times \mc H$ be a self-adjoint linear relation.
Then there exists a unique spectral measure $E$ on $\langle \bb R\cup\{\infty\},\mf B(\bb R\cup\{\infty\}), \mc H\rangle$,
such that 
\begin{equation*}
 	(A-z)^{-1}=\int_{\sigma(A)} \frac{1}{t-z} \, dE(t) 
\end{equation*}
for any $z\in\rho(A)$. In particular, we have $r(A)=\int_{\sigma(A)} r(t) \, dE(t)$ for any 
$r\in \bb C_{\rho(A)}(z)$.
\end{theorem}

This famous result gives rise to the above mentioned functional calculus assigning to each bounded and Borel measurable
function $\phi$ on the spectrum of $A$ -- recall that $\sigma(A)\subseteq \bb R\cup\{\infty\}$ -- the linear operator 
$\int \phi \, dE$. By the way, if $\phi$ is a characteristic function of a Borel set $B\subseteq \bb R\cup\{\infty\}$, 
then $\int \phi \, dE=E(B)$. Thus, the existence of the functional calculus yields the existence of spectral projections. 

In the situation of a self-adjoint definitizable linear relation $A$ on a Krein space $\mc K$ it turns out that, 
in general, we cannot consider all bounded and Borel measurable functions on $\sigma(A)$. We will have to take 
into account the zeros of the definitizing rational function $q$. According to \thref{definitzreal}, we can and 
will assume that $q$ is a fixed real rational function, i.e.\ $q^\#=q$.

In order to be able to include also the isolated zeros -- in particular all non-real zeros -- 
of the definitizing rational function, we are going to consider functions $\phi$ on the spectrum which have values  
$\phi(w)\in \bb C^{\mf d(w)+1}$, where $\mf d: \sigma(A) \to \bb N_0$ is the function, which assigns to 
$w\in \sigma(A)$ $q$'s degree of zero at $w$.  
Recall that $q$'s degree of zero at $\infty$ is 
$\max(0,\deg b - \deg a)$, where $a$ and $b$ are polynomials such that $q=\frac{a}{b}$. 
Also note that our assumption $q^\#=q$ implies $\mf d(w)=\mf d(\overline{ w})$ for all $w \in \sigma(A)$. 

Let us be more precise. First, we provide $\bb C^m$ with an algebraic structure. 

\begin{definition}\thlab{muldef1}
    For $x=(x_0,\dots,x_{m-1}),y=(y_0,\dots,y_{m-1})\in \bb C^m, \ \lambda\in \bb C$ let $x+y$ and $\lambda x$ be the usual
    componentwise addition and scalar multiplication. Moreover, we set
    \[
	x\cdot y := \big(\sum_{k=0}^j x_k y_{j-k}\big)_{j=0}^{m-1} \,,
    \]
    and $\overline{x}:=(\bar x_0,\dots,\bar x_{m-1})$.
\end{definition}

\begin{remark} \thlab{cminvert}
Note that $\bb C^m$ is a commutative $*$-algebra with multiplicative identity $(1,0,\ldots,0)$. 
Moreover, an element $x\in \bb C^m$ is multiplicatively invertible if and only if $x_0\neq 0$.
\end{remark}

\begin{definition} \thlab{muldef2}
   Let $\mc M(q,A)$ be the set of functions $\phi: \sigma(A) \to \dot \bigcup_{m\in \bb N} \bb C^m$ 
   such that $\phi(w) \in \bb C^{\mf d(w)+1}$. We provide $\mc M(q,A)$ pointwise with scalar multiplication, 
   addition $+$ and multiplication $\cdot$, where the operations on $\bb C^{\mf d(\lambda)+1}$ are
   as in \thref{muldef1}. For $\phi\in \mc M(q,A)$ we define $\phi^\#\in \mc M(q,A)$ by 
   $\phi^\#(\lambda) = \overline{\phi(\bar \lambda)}, \ \lambda\in \sigma(A)$. 
   
   By $\mc M_0(q,A)$ we denote the set of all function 
   $\phi\in \mc M(q,A)$ such that for all $w\in \sigma(A)$ all entries of $\phi(w)$ with the possible
   exception of the last one vanish, or equivalently that  
   $\phi(w) \in \{0\}\times \bb C^{1} \ (\subseteq \bb C^{\mf d(w)+1})$
   for all $w$ with $\mf d(w) > 0$.
   
   If $g: \sigma(A) \to \bb C$, i.e.\ $g\in \bb C^{\sigma(A)}$ and $\phi\in \mc M_0(q,A)$, 
   we define also
   $g\cdot \phi \in \mc M_0(q,A)$ pointwise, where the multiplication in each point is just the scalar
   multiplication.
\end{definition}

With the above introduced operations $\mc M(q,A)$ is a $*$-algebra.
Any function $f: \dom f \to \bb C$ with $\sigma(A)\subseteq \dom f$, 
which is holomorphic locally at $q^{-1}\{0\} \cap \sigma(A)$,
can be considered as an element $f_{q,A}$ of $\mc M(q,A)$ by the following procedure.
    
\begin{definition}\thlab{feinbetef}
We say that a function $f:\dom(f)\to \bb C$ is \emph{locally holomorphic} at a point $z\in \dom(f)$ 
if there is an open set $O\subseteq \bb C \cup \{\infty\}$ with $z\in O \subseteq \dom(f)$ 
such that $f |_O$ is holomorphic.

By $\operatorname{Hol}_{q^{-1}\{0\} \cap \sigma(A) }( \sigma(A))$ we denote the set of all functions $f:\dom(f)\to \bb C$
with $\sigma(A)\subseteq \dom(f)$ which are locally holomorphic at all points $q^{-1}\{0\}\cap \sigma(A)$.

We write $f \sim g$ for two functions $f,g\in \operatorname{Hol}_{q^{-1}\{0\} \cap \sigma(A)}(\sigma(A))$ 
if they coincide on $\sigma(A)$ and define the same germ at all points from $q^{-1}\{0\}\cap \sigma(A)$.

For $f\in \operatorname{Hol}_{q^{-1}\{0\} \cap \sigma(A)}( \sigma(A))$ we set $f_{q,A}(\lambda):=f(\lambda)$ for 
$\lambda \in \sigma(A)\backslash q^{-1}\{0\}$, i.e. $\mf d(w)=0$.
Otherwise, define 
    \begin{equation}\label{einbettdef}
    f_{q,A}(\lambda):=
    \left(\frac{f^{(0)}(\lambda)}{0!},\dots,\frac{f^{(\mf d(\lambda))}(\lambda)}{\mf d(\lambda)!}\right) 
	  \in \bb C^{\mf d(\lambda)+1}\,. 
\end{equation}
In the case that $\lambda =\infty$, all expressions of the form $f^{(j)}(\infty)$ have to be understood as 
$g^{(j)}(0)$ for $g(z):=f\left(\frac{1}{z}\right)$.
\end{definition}

      Obviously, $\sim$ is an equivalence relation and $f\sim g$ implies $f_{q,A}=g_{q,A}$. 
      Hence, it is possible to identify such functions and consider 
      $\operatorname{Hol}_{q^{-1}\{0\} \cap \sigma(A) }( \sigma(A)) /_{\sim}$ together with the 
      well-defined mapping $[f]_\sim \mapsto f_{q,A}$.
      
      Note that $\operatorname{Hol}_{q^{-1}\{0\} \cap \sigma(A)}(\sigma(A)) /_{\sim}$ is an algebra
      by introducing $[f]_\sim + [g]_\sim:=[f+g]_\sim$ and $[f]_\sim \cdot [g]_\sim:=[f\cdot g]_\sim$, 
      where the $f+g$ and $f\cdot g$ are defined on their common domain.

      Since $\sigma(A)$ and $q^{-1}\{0\}$ are symmetric with respect to $\bb R$,
      we can also define the mapping $f\mapsto f^\#$ by $\dom f^\# := \overline{\dom f}$ and $f^\#(z) := \overline{f(\bar z)}$ 
      acting on $\operatorname{Hol}_{q^{-1}\{0\} \cap \sigma(A)}(\sigma(A))$.
      It is easy to check that, $f \sim g$ implies $f^\# \sim g^\#$. Therefore,
      $[f]_\sim \mapsto [f]_\sim^\#:=[f^\#]_\sim$ is well-defined. 
      
      Equipped with these operations $\operatorname{Hol}_{q^{-1}\{0\} \cap \sigma(A)}(\sigma(A)) /_{\sim}$ 
      is a $*$-algebra.

\begin{lemma}\thlab{germeigpost}
The mapping 
\[
	\left\{ \begin{array}{rcl}
	\operatorname{Hol}_{q^{-1}\{0\} \cap \sigma(A)}(\sigma(A)) /_{\sim} & \to & \mc M(q,A) \\
	{[}f{]}_\sim & \mapsto & f_{q,A} \,,
	\end{array} \right.
\]
constitutes a $*$-algebra homomorphism.
\end{lemma}
\begin{proof}
We only show that it preserves multiplication. 

Let $f,g\in \operatorname{Hol}_{q^{-1}\{0\} \cap \sigma(A) }( \sigma(A))$. 
Trivially, we have $(f\cdot g)_{q,A} = f_{q,A}\cdot g_{q,A}$ on $\sigma(A)\backslash q^{-1}\{0\}$. 
The $j$-th entry of $(f\cdot g)_{q,A}(\lambda)$ at a zero $\lambda \in q^{-1}\{0\}$ is, by the general Leibniz rule, 
equal to
\[
	\sum_{k=0}^j \binom{j}{k}  \frac{1}{j!}  f^{(k)}(\lambda)  g^{(j-k)}(\lambda)=\sum_{k=0}^j  \frac{f^{(k)}(\lambda)}{k!}  
	\frac{g^{(j-k)}(\lambda)}{(j-k)!}	
	\,.
\]
This is nothing but the $j$-th entry of the product $f_{q,A}(\lambda)\cdot g_{q,A}(\lambda)$, cf.\ \thref{muldef1}.
\end{proof}

Clearly, for $s\in \bb C_{\rho(A)}(z)$ the assumptions in \thref{feinbetef} are satisfied. Hence,
$s_{q,A}$ is well defined and, as a consequence of \thref{germeigpost},
\[
	\left\{ \begin{array}{rcl}
	\bb C_{\rho(A)}(z) & \to & \mc M(q,A) \,, \\
	s & \mapsto & s_{q,A} \,,
	\end{array} \right.
\]
is a $*$-algebra homomorphism. 

Note also that for all $\lambda\in \sigma(A)$ exactly the last entry of
$q_{q,A}(\lambda) \in \bb C^{\mf d(w)+1}$ does not vanish, because $\lambda$ is a zero of $q$ of degree 
exactly $\mf d(w)$. In particular, $q_{q,A}\in \mc M_0(q,A)$.

\begin{lemma}\thlab{einbett2}
    For any $\phi\in \mc M(q,A)$ there exists an $s\in \bb C_{\rho(A)}(z)$ such that
    $\phi-s_{q,A}\in \mc M_0(q,A)$.
\end{lemma}
\begin{proof}
Let $\pi_{\mf d(w)}: \bb C^{\mf d(w)+1} \to \bb C^{\mf d(w)}$ denote the projection onto
    the first $\mf d(w)$ entries, and look at the linear map
\[
 \operatorname{J}:\left\{\begin{array}{rcl}
	\mc M(q,A) &\to & \prod_{w\in q^{-1}\{0\} \cap \sigma(A)} \bb C^{\mf d(w)} \,, \\
	\phantom{\Big(}\phi &\mapsto & \big(\pi_{\mf d(w)} \phi(w)\big)_{w\in q^{-1}\{0\} \cap \sigma(A)} \,.
	\end{array} \right.
\]
Clearly, $\ker \operatorname{J} = \mc M_0(q,A)$. For $\mu \in \rho(A)$ and $m:=\sum_{w\in q^{-1}\{0\} \cap \sigma(A)} \mf d(w)$ consider the linear space
\[
	\mc R:=\left\{ \frac{p(z)}{(z-\mu)^{m-1}} : p \in \bb C[z], \ p \ \text{ is of degree } < m \right\}.
\]
We claim that the restriction of $\operatorname{J}$ to $\mc R_{q,A}$ is bijective.

Assume $\operatorname{J}(s_{q,A}) = 0$ for $s(z)=p(z)(z-\mu)^{-(m-1)} \in \mc R$. Then, $s$ has zeros at all 
$w\in q^{-1}\{0\}\cap \sigma(A)$ with multiplicity at least $\mf d(w)$. Consequently, $p$ has zeros at all 
$w\in q^{-1}\{0\}\cap \sigma(A)\cap \bb C$ with multiplicity at least $\mf d(w)$. For $p\neq 0$ this entails 
that the degree of $p$ is greater or equal than $m-\mf d(\infty)$, where for reasons of convenience
we set $\mf d(\infty) = 0$ in the case that $\infty\not\in \sigma(A)$.

If $\mf d(\infty) = 0$, we therefore arrive at a contradiction to the definition of $\mc R$. 
Otherwise, $s$ having a zero at $\infty$ of multiplicity at least $\mf d(\infty)$ 
means that the degree of $p$ is at most $m-1-\mf d(\infty)$. The resulting inequality
\[
	m- \mf d(\infty) \leq \deg p \leq m-1 - \mf d(\infty)
\]
again is a contradiction. Thus, in any case we must have $p= 0$, and in turn $s = 0$.

Since the dimension of both spaces is $m$, this mapping is also bijective. For a given $\phi \in \mc M(q,A)$ there is a unique $s\in \mc R$ with $\operatorname{J}(s_{q,A})=\operatorname{J}(\phi)$. This just means $\phi-s_{q,A} \in \ker \operatorname{J}= \mc M_0(q,A)$.
\end{proof}

    %
    %

\begin{corollary}\thlab{Fform}
    Any function $\phi\in \mc M(q,A)$ admits a decomposition of the form
    \begin{equation}\label{uiwq37as}
	\phi = s_{q,A} + g\cdot q_{q,A} \,,
    \end{equation}
    where $s\in \bb C_{\rho(A)}(z)$ and $g: \sigma(A) \to \bb C$. 
\end{corollary}
\begin{proof}
    By \thref{einbett2} there exists a rational $s\in \bb C_{\rho(A)}(z)$ such that
    $h:=\phi-s_{q,A} \in \mc M_0(q,A)$.

    We define $g: \sigma(A) \to \bb C$ for
    $\lambda \in \sigma(A) \setminus q^{-1}\{0\}$ by $g(\lambda):=\frac{h(\lambda)}{q(\lambda)}$.
   For $\lambda \in q^{-1}\{0\} \cap \sigma(A)$
    we have 
		\[
			q_{q,A}(\lambda)_{\mf d(\lambda)}=\frac{q^{(\mf d(\lambda))}(\lambda)}{\mf d(\lambda)!}\neq 0,
		\]
		since 
    $\lambda$ is a zero of $q$ of degree exactly ${\mf d(\lambda)}$. Hence we can define
    \[
	g(\lambda):=\frac{h(\lambda)_{\mf d(\lambda)}}{q_{q,A}(\lambda)_{\mf d(\lambda)}} \,.
    \]
    It is then easy to verify that $\phi = s_{q,A} + g\cdot q_{q,A}$.

\end{proof}

This decomposition is by no means unique, since already the choice of the pole $\mu$ of the rational function $s$ was arbitrary. 
In an important special case, there is a canonical decomposition:

\begin{remark}\thlab{polyrem}
    If $A$ is an everywhere defined and bounded operator, i.e.\ $\infty\in \rho(A)$, then
    we  could have taken also $\mc R=\{ p(z) : p \in \bb C[z], \, p \ \text{ is of degree } < m \}$ in the proof
    of \thref{einbett2}. Thus, for any $\phi\in \mc M(q,A)$ there exists a polynomial $p(z)\in \mathbb C[z]$ of degree
    less than $m$ such that
    $\phi-p_{q,A}\in \mc M_0(q,A)$. Since the linear mapping $\operatorname{J}$ constitutes a bijection from 
    $\mc R_{q,A}$ onto $\prod_{w\in q^{-1}\{0\} \cap \sigma(A)} \bb C^{\mf d(w)}$, we even get a unique 
    polynomial $p(z)\in \mathbb C[z]$ of degree less than $m$ such that $\phi-p_{q,A}\in \mc M_0(q,A)$.
    
    Using this for the proof of \thref{Fform} we see that any function $\phi\in \mc M(q,A)$ 
    admits a unique decomposition of the form $\phi = p_{q,A} + g\cdot q_{q,A}$,
    where $p(z)$ is a polynomial of degree
    $<m$ and $g: \sigma(A) \to \bb C$.
\end{remark}

Not all functions from $\mc M(q,A)$ can be integrated and then give rise to an operator.
In fact, just those $\phi\in \mc M(q,A)$ which admit a decomposition as in \thref{Fform} with a 
bounded and Borel measurable function
$g:\sigma(A)\to \bb C$ shall be useful for our purposes.

\begin{definition}\thlab{FdefklM}
    By $\mc F(q,A)$ we denote the set of all $\phi\in \mc M(q,A)$ 
    such that for some decomposition \eqref{uiwq37as} the function $g$ is bounded and Borel measurable, i.e.\
    \[
	\mc F(q,A) := \{ s_{q,A} + g q_{q,A} : s\in \bb C_{\rho(A)}(z), \, g\in \mf B(\sigma(A)) \} \,,
    \]
    where $\mf B(\sigma(A))$ denotes the algebra of all complex valued, 
    bounded and Borel measurable functions on $\sigma(A)$.
\end{definition}

It is elementary to verify that $\mc F(q,A)$ is a $*$-subalgebra of 
$\mc M(q,A)$. Trivially, we have $\bb C_{\rho(A)}(z)_{q,A} \subseteq \mc 
F(q,A)$.

\begin{remark}\thlab{unabhvq}
The following result, \thref{zerlgeig}, contains a more explicit criterion whether a function belongs to $\mc F(q,A)$ 
or not. In particular, \thref{FdefklM} does not depend on the concrete decomposition.

This criterion also shows that $\mc F(q,A)$ only depends on $\sigma(A)$ and the zeros and their multiplicity of $q$.
The same is by definition true for the classes $\mc M(q,A)$ and $\mc M_0(q,A)$.
\end{remark}

\begin{remark} \thlabel{unendlich}
When considering zeros of $q$ in $\sigma(A)$, one possibly has to deal with the point $\infty$, which needs some special treatment. The difference to a common zero of $q$ in 
$\sigma(A)\cap \bb C$ is not fundamental but only affects the notation.

Basically the same statements and proofs formulated for complex zeros also hold true at $\infty$.
One simply has to substitute $(z-\infty)^{-n}:=z^n$ for $z \in \bb C$ and $n \in \bb Z$. 
\end{remark}

\begin{lemma}\thlab{zerlgeig}
    Let $\phi\in \mc M(q,A)$ satisfy $\phi = s_{q,A} + g\cdot q_{q,A}$ with 
    $s\in \bb C_{\rho(A)}(z)$ and $g: \sigma(A) \to \bb C$.
		
    Then, the Borel measurability of $g$ is equivalent to the Borel measurability of 
    $\phi\vert_{\sigma(A)\setminus q^{-1}\{0\}}$.	
    Moreover, $g$ is bounded if and only if both $\phi\vert_{\sigma(A)\setminus q^{-1}\{0\}}$ is bounded and
		\footnote{For $w=\infty$, this reads as 
		$\lambda^{\mf d(\infty)} \left( \phi(\lambda) - \sum_{j=0}^{\mf d(\infty)-1} \phi(\infty)_j 
		\frac{1}{\lambda^j} \right)$; see \thref{unendlich}.} 
    \begin{equation}\label{fn8qw3}
	\frac{1}{(\lambda-w)^{\mf d(w)}}\left(\phi(\lambda) - \sum_{j=0}^{{\mf d(w)-1}} \phi(w)_j (\lambda-w)^j\right), 
	\ \ \lambda \in U(w)\setminus \{w\} \,,
    \end{equation}
    is bounded for some neighbourhood $U(w)$ of $w$ with respect to the relative \mbox{topology} on $\sigma(A)$ 
    for each $w \in q^{-1}\{0\}\cap \sigma(A)$ which is not isolated in $\sigma(A)$, 
    i.e. $w\in \Cl(\sigma(A) \setminus \{ w \} )$.

In this case, \eqref{fn8qw3} is bounded on all neighbourhoods $U(w)$ of $w$ for all non-isolated $w\in q^{-1}\{0\} \cap \sigma(A)$, 
as long as $\Cl(U(w)) \cap q^{-1}\{0\} = \{w \} $.
\end{lemma}

\begin{proof}
    Any $r\in \bb C_{\rho(A)}(z)$ is continuous on the compact set $\sigma(A)$. Hence $r$, and thereby 
    the restriction $r\vert_{\sigma(A)\setminus q^{-1}\{0\}} = r_{q,A}\vert_{\sigma(A)\setminus q^{-1}\{0\}}$ is 
    measurable and bounded.
		
    Since $q$ does not vanish on $\sigma(A)\setminus q^{-1}\{0\}$, it follows from 
		\[
			\phi\vert_{\sigma(A)\setminus q^{-1}\{0\}} = s\vert_{\sigma(A)\setminus q^{-1}\{0\}} +
    g\vert_{\sigma(A)\setminus q^{-1}\{0\}} \ q\vert_{\sigma(A)\setminus q^{-1}\{0\}}
		\]
    that  $\phi\vert_{\sigma(A)\setminus q^{-1}\{0\}}$ is measurable if and only if 
    $g\vert_{\sigma(A)\setminus q^{-1}\{0\}}$ has this property.
    As $q^{-1}\{0\}$ is finite, and hence a Borel set, this is equivalent to the Borel measurability of $g$.

 Concerning the boundedness, consider for a non-isolated zero $w\in q^{-1} \{0\}\cap \sigma(A)$ the rational function 
		\begin{equation} \label{nurrat01}
			 \lambda \mapsto \frac{1}{(\lambda - w)^{\mf d(w)}}\left(s(\lambda) - \sum_{j=0}^{\mf d(w)-1} 
	    \phi(w)_j (\lambda-w)^j\right).
		\end{equation}
		Note that $\phi - s_{q,A}\in \mc M_0(q,A)$, i.e.  $j! \, \phi(w)_j = s^{(j)}(w)$ for $j =0,\dots,\mf d(w)-1$. 
		Taylor expansion of the holomorphic function $s$ around $w$ makes clear that \eqref{nurrat01} has no pole at $\lambda=w$.
		Hence all poles of this rational function lie in $\rho(A)$. 
		In particular, \eqref{nurrat01} is bounded on $\sigma(A)$.

		For any non-isolated $w\in  q^{-1}\{0\} \cap \sigma(A)$ and for $\lambda \in \sigma(A)\backslash q^{-1}\{0\}$ we have
		
		
    \begin{align*}
	\frac{q(\lambda)}{(\lambda - w)^{\mf d(w)}} 	g(\lambda) &= \frac{\phi(\lambda) - s_{q,A}(\lambda)}{(\lambda - w)^{\mf d(w)}} \\     
 &= \frac{1}{(\lambda - w)^{\mf d(w)}} \left( \phi(\lambda) - \sum_{j=0}^{\mf d(w)-1} \phi(w)_j (\lambda-w)^j\right)  \\
& \hspace{5mm} \ \ -\frac{1}{(\lambda - w)^{\mf d(w)}} \left( s(\lambda) - \sum_{j=0}^{\mf d(w)-1} \phi(w)_j (\lambda-w)^j  \right) \,.
	  \end{align*}

		Since $\mf d(w)$ denotes the degree of zero of $q$ at $w$, the expression $\frac{q(\lambda)}{(\lambda-w)^{\mf d(w)}}$ is both bounded 
		and bounded away from zero on any neighbourhood $U(w)$ of $w$ with $\Cl(U(w)) \cap q^{-1}\{0\} = \{w \}$.

    Hence $g$ is bounded on such a neighbourhood $U(w)$ if and only if the expression 
    in \eqref{fn8qw3} is bounded on $U(w)$. Note that this also implies the boundedness of 
    $\phi\vert_{\sigma(A)\setminus q^{-1}\{0\}}$ on $U(w)\setminus \{w\}$.
    
    For each non-isolated $w\in q^{-1} \{0\}\cap \sigma(A)$ we fix an open neighbourhood $U(w)$ of the
    indicated kind, 
    and set $U(w) := \{w\}$ for all isolated $w\in q^{-1}\{0\} \cap \sigma(A)$. Clearly, we have 
    $\phi(\lambda) - s(\lambda) = g(\lambda) q(\lambda)$
    for $\lambda \in \sigma(A)\setminus \bigcup_{w\in q^{-1}\{0\} \cap \sigma(A)} U(w)$. 

    Note that $q$ is not only bounded but also bounded away from zero on $\sigma(A)\setminus \bigcup_{w} U(w)$. 
    Hence $g$ is bounded on $\sigma(A)\setminus \bigcup_{w} U(w)$ if and only if $\phi - s$ is bounded there. 
    Since $s$ is bounded on $\sigma(A)$, this is also equivalent to the fact that $\phi$ is bounded on 
    $\sigma(A)\setminus \bigcup_{w} U(w)$.
\end{proof}

\begin{remark}\thlab{mcscheis}
\begin{sloppypar}
  For $\phi \in \mc M(q,A)$ and a non-isolated zero $w\in q^{-1}\{0\}\cap \sigma(A)$ the fact that 
  \eqref{fn8qw3} is bounded on a neighbourhood $U(w)$ of $w$ implies that $\phi(w)_0 ,\ldots , \phi(w)_{\mf d(w)-1}$ 
  are uniquely determined by the values of $\phi$ on $U(w)\setminus \{w \}$.
\end{sloppypar}

In particular, there is at most one choice of $\phi(w)_0 ,\ldots , \phi(w)_{\mf d(w)-1}$ such that 
$\phi \in \mc F(q,A)$. 
\end{remark}

Next, we give a readable sufficient condition for a function to belong to $\mc F(q,A)$.

\begin{corollary}\thlab{germbddm}
    For any function $f\in \operatorname{Hol}_{q^{-1}\{0\} \cap \sigma(A)}(\sigma(A))$ which is bounded and 
    measurable on $\sigma(A)$, the function  $f_{q,A}$ belongs to $\mc F(q,A)$.
\end{corollary}
\begin{proof}
    According to \thref{zerlgeig} we have to show that for $\phi=f_{q,A}$ the expression \eqref{fn8qw3} is bounded for
    a sufficiently small neighbourhood $U(w)$ of any non-isolated $w\in q^{-1}\{0\} \cap \sigma(A)$.
    But this is granted by the assumption that $f$ is analytic on a neighbourhood of $w$, 
    i.e.\ $f(z) = \sum_{n=0}^\infty \frac{f^{(n)(w)}}{n!} (z-w)^n$ for $|z-w|$ sufficiently small, and 
    by $\phi(w)=f_{q,A}(w) = \big(\frac{f^{(0)}(w)}{0!},\dots,\frac{f^{(\mf d(w))}(w)}{\mf d(w)!}\big)$.
\end{proof}

\begin{remark}\thlab{germeig2}
    If we denote by $\operatorname{Hol}^{bm}_{q^{-1}\{0\} \cap \sigma(A)}(\sigma(A))$ the set of all functions 
    $f\in \operatorname{Hol}_{q^{-1}\{0\} \cap \sigma(A)}(\sigma(A))$ which are bounded and 
    measurable on $\sigma(A)$, then $\operatorname{Hol}^{bm}_{q^{-1}\{0\} \cap \sigma(A)}(\sigma(A)) /_{\sim}$
    is a $*$-subalgebra of $\operatorname{Hol}_{q^{-1}\{0\} \cap \sigma(A)}(\sigma(A)) /_{\sim}$; see 
    \thref{feinbetef} and the considerations after.
    
    Moreover, as a restriction of the $*$-homomorphism in \thref{germeigpost}, the mapping
    $[f]_\sim\mapsto f_{q,A}$ from $\operatorname{Hol}^{bm}_{q^{-1}\{0\} \cap \sigma(A)}(\sigma(A)) /_{\sim}$
    to $\mc F(q,A)$ is a $*$-homomorphism.
\end{remark}

\begin{lemma}\thlab{calculwohldef}
Let $\mf B(\sigma(A))$ denotes the algebra of all complex valued, 
bounded and measurable functions on $\sigma(A)$. We endow the space
$\bb C_{\rho(A)}(z) \times \mf B(\sigma(A))$ with componentwise addition and scalar multiplication, and with 
\[
	(s_1,g_1)\cdot (s_2,g_2):=(s_1 s_2,s_1g_2 + s_2g_1 + q g_1 g_2 ).
\]
Equipped with $(s,g)^\#:=(s^\#,g^\#)$, this space becomes a $*$-algebra.
The mapping 
\[
	\omega : \left\{ \begin{array}{rcl}
	\bb C_{\rho(A)}(z) \times \mf B(\sigma(A)) & \to & \mc M(q,A) \,, \\
	(s,g) & \mapsto & s_{q,A} + g\cdot q_{q,A} \,,
	\end{array} \right.
\]
constitutes a $*$-algebra homomorphism with $\ran \omega = \mc F(q,A)$.
Moreover, we have 
\[
    (s,g) \in \ker \omega \ \ \text{ if and only if } \ \ 
    g=-\big( \frac{s}{q} \big)\big\vert_{\sigma(A)} \,.
\]
\end{lemma}
\begin{proof}
It is elementary to verify that $\omega$ is a $*$-homomorphism. 

In order to proof the last claim, consider $(s,g) \in \ker \omega$, i.e. $s_{q,A} + g\cdot q_{q,A}=0$. 
For $\lambda \in \sigma(A)\setminus q^{-1}\{0\}$ this just means
    $g(\lambda)=-\frac{s(\lambda)}{q(\lambda)}$.
	At a zero $w \in q^{-1}\{0\} \cap \sigma(A)$ due to $q_{q,A}\in \mc M_0(q,A)$ we get 
	\[
		s^{(j)}(w) = 0, \quad j=0,\ldots, \mf d(w) -1 \,,
	\]
    and $s^{(\mf d(w))}(w)=\mf d(w)! \ s_{q,A}(w)_{\mf d(w)}= - g(w) q^{(\mf d(w))}(w)$, 
		which yields
    \[
	g(w)=- \frac{s^{(\mf d(w))}(w)}{q^{(\mf d(w))}(w)} =
	\lim_{\lambda \to w} -\frac{s(\lambda)}{q(\lambda)} \,.
    \]
		Hence, any $w\in q^{-1}\{0\} \cap \sigma(A)$ is a zero of $s$ with the 
		same or higher multiplicity, and $g$ is the rational function $-\frac{s}{q}$ restricted to $\sigma(A)$. 
\end{proof}

\begin{remark}\thlab{bicommrem}
  For a linear relation $A$ we denote by  
  $A'$ the space of all $C\in B(\mc K)$ such that $(C\times C)(A) \subseteq A$.
  By \thref{comminter} this coincides with the usual commutant if $A\in B(\mc K)$.
  In turn, $A''$ denotes the commutant of $A'$.

  Since $(A-z)^{-1} = \tau_M(A)$ for a proper $M\in \bb C^{2\times 2}$, relation \eqref{zzue4} 
  shows that $C \in A'$ to $(C\times C)((A-z)^{-1}) \subseteq (A-z)^{-1}$.
  If $A$ has a non-empty resolvent set, then 
  by \thref{comminter} this inclusion holds if and only if $C (A-z)^{-1} = (A-z)^{-1} C$ 
  for one, and hence for all $z\in \rho(A)$, i.e.\ $A'=((A-z)^{-1})'$.
\end{remark}

\noindent
In the following theorem, we establish our functional calculus.

\begin{theorem}\thlab{zwe4578}
  Let $A$ be a self-adjoint definitizable linear relation on a Krein space $\mc K$, and let $q$ be a 
  real definitizing rational function. 

  Let $\mc V$ be a Hilbert space and $T: \mc V \to \mc K$ 
  be bounded linear and injective such that $TT^+ = q(A)$; see \thref{ht20g95}.
  Moreover, denote by $E$ the spectral measure of the self-adjoint linear relation 
  $\Theta(A) = (T\times T)^{-1}(A)$ on the Hilbert space $\mc V$; see \thref{umktheta} and \thref{ht20g95rem2}.
  
  For any $\phi\in \mc F(q,A)$ let $\phi=s_{q,A} + g\cdot q_{q,A}$ be a decomposition as in \thref{Fform}. Then
  \[
      \phi(A) := s(A) + T \ \int_{\sigma(\Theta(A))} g \, dE \ T^+ \,,
  \]
  is a bounded operator in $A''$ which does not depend on the decomposition of $\phi$. 
  The map $\phi \mapsto \phi(A)$, $\mc F(q,A) \to A''\subseteq B(\mc K)$ constitutes a $*$-homomorphism which 
  extends the rational functional calculus. 
\end{theorem}
\begin{proof}
Denoting again by $\mf B(\sigma(A))$ the algebra of all bounded and measurable functions on $\sigma(A)$,
consider the mapping $\Psi: \bb C_{\rho(A)}(z) \times \mf B(\sigma(A)) \to B(\mc K)$ defined by
\[
    \Psi(s,g) := s(A) + \Xi \Big(\int_{\sigma(\Theta(A))} g \, dE\Big) \,,
\]
where $\Xi$ is as in \thref{Xidefeig}. Note that $\int g \, dE$ belongs to $(T^+T)'$
since $T^+T= \Theta (q(A)) = q(\Theta(A)) = \int q \, dE$.

\noindent
For $(s_1,g_1), (s_2,g_2) \in \bb C_{\rho(A)}(z) \times \mf B(\sigma(A))$ we have (see \thref{calculwohldef})
\begin{multline}\label{htide4}
  \Psi\big( (s_1,g_1)\cdot (s_2,g_2) \big) = \\ s_1(A)s_2(A) + 
    \Xi\Big( \int_{\sigma(\Theta(A))} s_1g_2 + s_2g_1 + q g_1 g_2 \, dE \Big) \,.
\end{multline}
Since $\int_{\sigma(\Theta(A))} r \, dE = r(\Theta(A)) = \Theta(r(A))$ for $r\in \bb C_{\rho(A)}(z)$,
we obtain from \thref{Xidefeig} that \eqref{htide4} coincides with
\[
    \left( s_1(A) + \Xi \Big(\int_{\sigma(\Theta(A))} g_1 \, dE\Big) \right) 
	    \left( s_2(A) + \Xi \Big(\int_{\sigma(\Theta(A))} g_2 \, dE\Big) \right) \,.
\]
It is straight forward to check that $\Psi$ is also linear and satisfies $\Psi(s^\#,g^\#) = \Psi(s,g)^+$.
Hence it is a $*$-homomorphism. 

By the definition of the rational functional calculus it is obvious that $s(A) \in ((A-z)^{-1})'' = A''$, where 
$z\in \rho(A)$; see \thref{bicommrem}. On the other hand, it is a well known property of the 
Spectral Theorem on Hilbert spaces that $\int g \, dE$ 
commutes with any operator from $\big((\Theta(A)-z)^{-1}\big)'$. Since 
$((\Theta(A)-z)^{-1})' \subseteq q(\Theta(A))' = (T^+T)'$, we conclude from \thref{Xidefeig}
that
\[
    \Xi\Big(\int_{\sigma(\Theta(A))} g \, dE\Big) \in (TT^+)'' = q(A)'' \subseteq A'' \,,
\]
and see $\Psi(s,g)\in A''$.
Finally,
recall that for $(s,g) \in \ker \omega$ we have $g=-(\frac{s}{q})\vert_{\sigma(A)}$, due to \thref{calculwohldef}. 
With the help of \thref{umktheta}, \thref{Xidefeig} and \thref{ratcalcul} we see that
\begin{align*}
    \Psi(s,g) &= s(A) - \Xi \Big(\int_{\sigma(\Theta(A))}  \left(\frac{s}{q}\right)\Big\vert_{\sigma(A)} \, dE\Big) = 
    s(A) - \Xi\Big( \,\frac{s}{q} \big( \Theta(A)\big) \Big) \\ & = 
		s(A) - \Xi\Big( \Theta\big( \, \frac{s}{q}(A)\big) \Big) 	= s(A) - TT^+ \frac{s}{q}(A) = s(A) - q(A) \frac{s}{q}(A) = 0 \,,
\end{align*}
i.e.\ $\ker \omega \subseteq \ker \Psi$. Therefore, $\Psi \circ \omega^{-1}: \mc F(q,A) \to A''$ is a well 
defined $*$-homomorphism.
\end{proof}

We equip $\mc F(q,A)$ with a norm and state the continuity of our functional calculus.

\begin{definition}
For all non-isolated zeros $w\in q^{-1}\{0\} \cap \sigma(A)$ let $U(w)\subseteq \sigma(A)$ be a fixed neighbourhood of 
$w$ with respect to $\sigma(A)$'s relative topology such that $\Cl(U(w)) \cap q^{-1}\{0\} = \{w\}$.
We declare a norm on $\mc F(q,A)$ by
\begin{multline*}
		\| \phi \|_{\mc F} := \sup_{\lambda \in \sigma(A)\backslash q^{-1}\{0\} } \vert \phi(\lambda) \vert \,
		+    \sum_{w \in q^{-1}\{0\}\cap \sigma(A)} \| \phi(w)\|_\infty \, + \\ 
		 +  \sum_{\stackrel{w \in q^{-1}\{0\}\cap \sigma(A)}{w \ \text{non-isolated}}} 
			\sup_{\lambda \in U(w)} \frac{1}{| \lambda - w |^{\mf d(w)}} \left\vert \phi(\lambda) 
				  - \sum_{j=0}^{\mf d(w)-1} \phi(w)_j (\lambda -w)^j  \right\vert \,.
\end{multline*}
Here, $\|.\|_\infty$ denotes the maximum norm of a vector in $\bb C^{\mf d(\omega)+1}$.
\end{definition}

\begin{theorem} \thlab{calculstetig}
Let $A$ be a self-adjoint definitizable linear relation on a Krein space $\mc K$. Equip $\mc F(q,A)$ 
with the norm $\|.\|_{\mc F}$, and endow the space of all bounded operators, $B(\mc K)$, with the operator norm.

Then, the functional calculus $\mc F(q,A) \to B(\mc K), \phi \mapsto \phi(A)$, defined in \thref{zwe4578}, is continuous.
\end{theorem}
\begin{proof}
We start by making this mapping more explicit. Recall the definition of $\mc R$ in the proof of \thref{einbett2},
where now we assume that the pole $\mu$ of the functions from $\mc R$ is contained in 
$\bb C\setminus (\bb R \cup q^{-1}\{0\})$.
We claim that the restriction of $\omega:  \bb C_{\rho(A)} \times \mf B(\sigma(A)) \to \mc F(q,A)$  to 
$\mc R \times \mf B(\sigma(A))$ is bijective:

Take $(s;g)\in [\mc R \times \mf B(\sigma(A)) ] \cap \ker \omega$. 
\thref{calculwohldef} gives $g=-\big(\frac{s}{q}\big)\big\vert_{\sigma(A)}$, 
where all zeros of $q$ in $\sigma(A)$ need to be zeros of $s$ with the same or higher multiplicity. 
With the notation introduced in the proof of \thref{einbett2}, we have $\operatorname{J}(s_{q,A})=0$. In 
this proof we showed that $\operatorname{J}$ is injective on $\mc R_{q,A}$, which gives $s=0$ and in turn $(s;g)=0$.

Our functional calculus can therefore be written as $\Psi \circ \omega\vert_{\mc R \times \mf B(\sigma(A))}^{-1}$.
The linear map $\Psi: \mc R \times \mf B(\sigma(A)) \to B(\mc K)$ is continuous with respect to the first 
component for any norm on $\mc R$, because $\mc R$ is only finite dimensional. When we endow $\mf B(\sigma(A))$ with the 
supremum norm on $\sigma(A)$, the map $\Psi$ is also continuous with respect to the second component.

We are left to show that $\omega\vert_{\mc R \times \mf B(\sigma(A))}^{-1}: \mc F(q,A) \to \mc R \times \mf B(\sigma(A))$ 
is continuous. For $\phi \in \mc F(q,A)$ we write $\omega\vert_{\mc R \times \mf B(\sigma(A))}^{-1}(\phi) = (s;g)$ with 
$\phi = s_{q,A} + g \cdot q_{q,A}$. Hereby, $s$ is the rational function in $\mc R$ such that 
$\phi-s_{q,A}\in \mc M_0(q,A)$, as described in \thref{einbett2} and \thref{Fform}. In particular, $s$ only 
depends on the values of $\phi(w)$ for $w\in q^{-1}\{0\}\cap \sigma(A)$. 

Consider again the linear map 
$\operatorname{J}: \mc M(q,A) \to \prod_{w\in q^{-1}\{0\} \cap \sigma(A)} \bb C^{\mf d(w)}$, introduced 
in the proof of \thref{einbett2}, which maps  $\phi\in \mc M(q,A)$ to the tuple which contains for all 
$w\in q^{-1}\{0\}\cap \sigma(A)$ the first $\mf d(w)$ entries of $\phi(w)$, 
${\operatorname{J}(\phi)=\big(\pi_{\mf d(w)} \phi(w)\big)_{w\in q^{-1}\{0\} \cap \sigma(A)}}$. 
By definition of the norm on $\mc F(q,A)$, $\operatorname{J}\vert_{\mc F(q,A)}$ is continuous. 

Also, assigning to each vector $x\in \prod_{w\in q^{-1}\{0\} \cap \sigma(A)} \bb C^{\mf d(w)}$ the corresponding rational function $s \in \mc R$ is linear and continuous, since the domain is only finite dimensional.
This shows that $\mc F(q,A)\to \mc R, \, \phi \mapsto s$ is continuous. 
In particular, if we choose on $\mc R$ a specific norm, namely
($l=\max\{ \mf d(w) : w \in q^{-1}\{0\}\cap \sigma(A)\}$)
\[
    \|s\|_{\mc R} := 
	\sum_{k=0}^l (\|s^{(k)}\|_{\infty, \bb R \cup \{\infty\}} + 
	\|t^{(k)}\|_{\infty, \bb R \cup \{\infty\}}) + \sum_{w \in q^{-1}\{0\}\cap \sigma(A)} 
	      |s^{(\mf d(w))}(w)| 
\]	      
where $t(z)=s(\frac{1}{z})$ and $\|.\|_{\infty,\bb R \cup \{\infty\}}$ 
denotes the supremum norm on $\bb R \cup \{\infty\}$, 
we get for a certain constant $C>0$
\begin{equation} \label{contis}
	\|s\|_{\mc R} \leq C \| \phi \|_{\mc F}.
\end{equation}

It is left to show that $\phi \mapsto g$ is continuous, i.e. $\|g \|_\infty \leq D \| \phi \|_{\mc F}$ for some constant $D>0$. 
We distinguish three different cases:

First, we look at the value of $g$ at a zero of $q$. For $w\in  q^{-1}\{ 0\} \cap \sigma(A)$, we have
\begin{align*}
	|g(w)| & = \Bigg\vert  \frac{ \phi(w)_{\mf d(w)} -s_{q,A}( w)_{\mf d(w)} }{ q_{q,A}( w)_{\mf d(w)}  }     \Bigg\vert 
	= \Bigg\vert \frac{ {\mf d(w)!} \phi(w)_{\mf d(w)} - s^{(\mf d(w))}(w)}{q^{(\mf d(w))}(w) }   \Bigg\vert \\ & \leq
	d \Big( \left\vert \phi(w)_{\mf d(w)} \right\vert + \big\vert s^{(\mf d(w))}(w) \big\vert \Big) \leq 
	d \left( \| \phi(w) \|_\infty + \|s\|_{\mc R} \right)  \leq D_1 \|  \phi \|_{\mc F},
\end{align*}
with $d:=\frac{\mf d(w)!}{\vert q^{(\mf d(w))}(w) \vert}$ and a constant $D_1>0$. The last inequality used \eqref{contis}.

Clearly, for any other $\lambda \in \sigma(A) \backslash q^{-1}\{0\}$, we have
\begin{equation} \label{ehklor}
	 g(\lambda)  =    \frac{\phi(\lambda)-s(\lambda)}{q(\lambda)} .
\end{equation}

Secondly, let $w\in \sigma(A)$ be a non-isolated, in particular real, zero of $q$, and consider 
the value of $g$ at those $\lambda\in \sigma(A)$ which are near $w$ , i.e.\ 
$\lambda \in U(w)\setminus\{w\} \subseteq \bb R \cup \{\infty\}$.
The estimate \footnote{Again, $\vert (\lambda - w)^{\mf d(w)} \vert$ has to be interpreted as 
$\vert \lambda^{-\mf d(\infty)} \vert$ if $w=\infty$; see \thref{unendlich}.} 
$\vert q(\lambda) \vert \geq c \vert (\lambda - w)^{\mf d(w)} \vert$, which holds for all $\lambda \in U(w)$, gives
\begin{multline*}
		|g(\lambda)| \leq     \frac{\left\vert\phi(\lambda)-s(\lambda)\right\vert}{c \vert (\lambda - w)^{\mf d(w)} \vert}  
		\leq  \frac{1}{c \vert (\lambda - w)^{\mf d(w)} \vert}    
		\left\vert   \phi(\lambda) - \sum_{j=0}^{\mf d(w)-1} \phi(w)_j (\lambda -w)^j \right\vert + 	\\ 
		\hspace{5mm} + \frac{1}{c \vert (\lambda - w)^{\mf d(w)} \vert}	\left\vert \sum_{j=0}^{\mf d(w)-1} \phi(w)_j (\lambda -w)^j-s(\lambda)  \right\vert.
\end{multline*}
The first summand is bounded by $c^{-1} \| \phi \|_{\mc F}$ by definition. 
Due to $\phi(w)_j=s_{q,A}(w)_j=\tfrac{1}{j!} s^{(j)}(w)$, the second summand turns out to be the remainder of the 
Taylor approximation of $s$, which for $w\neq \infty$ can be estimated by
\[
	\frac{1}{c\vert (\lambda - w)^{\mf d(w)} \vert}	
	\left\vert \sum_{j=0}^{\mf d(w)-1} \phi(w)_j (\lambda -w)^j-s(\lambda)  \right\vert \leq 
  \frac{1}{c \, \mf d(w)!}   \big\| s^{\mf d(w)} \big\|_{\infty,\bb R\cup\{\infty\}} 
\]
and for $w=\infty$ by $(t(z)=s(\frac{1}{z}))$
\[
	\frac{|\lambda^{\mf d(\infty)}|}{c}	
	\left\vert \sum_{j=0}^{\mf d(\infty)-1} \phi(\infty)_j \frac{1}{\lambda^j} - s(\lambda)  \right\vert \leq 
  \frac{1}{c \, \mf d(\infty)!}  \big\| t^{\mf d(\infty)} \big\|_{\infty,\bb R\cup\{\infty\}} 
\]
for all $\lambda \in U(w)\setminus \{w\} \ (\subseteq \bb R \cup \{\infty\})$.
Due to \eqref{contis},  this gives $\vert g(\lambda) \vert \leq D_2 \| \phi \|_{\mc F} $ 
for all $\lambda \in {\bigcup_{w} U(w)\setminus \{w\}}$ and for a constant $D_2>0$.

Finally, consider $\lambda \in \sigma(A) \backslash \left(\bigcup_{w} U(w) \cup q^{-1}\{0\}\right)$. In this case, 
there is a constant $c'>0$ such that the uniform lower estimate $|q(\lambda)|\geq c'$ holds.  
With \eqref{ehklor} and \eqref{contis}, we get for some constant $D_3>0$
\[
	|g(\lambda)| \leq \frac{1}{c'} \big( |\phi(\lambda)| +  |s(\lambda)| \big) \leq D_3 \| \phi \|_{\mc F}.
\]

We set $D:=\max (D_1,D_2,D_3)$, and conclude $\|g \|_\infty \leq D \| \phi \|_{\mc F}$.
\end{proof}

We have some a priori knowledge of the location of the spectrum of $\phi(A)$. Let $\pi_1$ denote the projection 
onto the first component,
\[
	\pi_1: \left\{ \begin{array}{rcl}
	\dot \bigcup_{m\in \bb N} \bb C^m & \to & \bb C \,, \\
	x &\mapsto &x_0 \,.
	\end{array} \right.
\]

\begin{proposition} \label{gehtEinfacher}
For all $\phi \in \mc F(q,A)$, we have
\[
	\sigma(\phi(A)) \subseteq \Cl\big((\pi_1 \circ \phi ) (\sigma(A))\big).
\]
\end{proposition}
\begin{proof}
Fix $\mu\notin \Cl\big((\pi_1 \circ \phi) (\sigma(A))\big) $ and consider the function $\psi\in \mc M(q,A)$ defined as 
\[
	\psi(\lambda):=\frac{1}{{\phi(\lambda)- \mu}} \ \ \text{ for } \ \ \lambda \in \sigma(A)\backslash q^{-1}\{0\}.
\]
We claim that indeed $\psi \in \mc F(q,A)$, when $\psi$ is defined correctly at the non-isolated zeros of $q$. 
If such a choice is possible, these values are uniquely determined, as already mentioned in \thref{mcscheis}.

As we want the identity $\psi(\lambda)\cdot(\phi(\lambda)-\mu \mathds{1}_{q,A}(\lambda))=\mathds{1}_{q,A}(\lambda)$ 
to hold for all $\lambda\in \sigma(A)$, we have to set $\psi(w)$, for a zero $w\in q^{-1}\{0\}\cap \sigma(A)$, 
as the multiplicative inverse of $\phi(w)-\mu \mathds{1}_{q,A}(w)\in \bb C^{\mf d(w)}$. As mentioned in \thref{cminvert}, 
a vector in $\bb C^{\mf d(w)}$ is invertible if and only if its first entry is not zero. 
By the choice of $\mu$, we have $\pi_1 (\phi(w)-\mu \mathds{1}_{q,A}(w)) = \pi_1(\phi(w)) - \mu \neq 0$ and can 
define for $w\in q^{-1}\{0\}\cap \sigma(A)$
\[
	\psi(w):= \left(  \phi(w)-\mu \mathds{1}_{q,A}(w) \right)^{-1}.
\]

We are going to apply \thref{zerlgeig} to the function $\psi$. 
First, we clearly find a neighbourhood of $\mu$ which is disjoint to $(\pi_1 \circ \phi) (\sigma(A))$.
Thus, $\lambda \mapsto  \phi(\lambda)-\mu  $ is bounded away from zero on $\sigma(A)\backslash q^{-1}\{0\}$. In particular, $\psi\vert_{\sigma(A)\setminus q ^{-1}\{0\}}$ is bounded and measurable.

In order to show that expression \eqref{fn8qw3} is bounded, let $w$ be a non-isolated zero of $q$ in $\sigma(A)$, 
and for $\lambda \in \sigma(A)\setminus q^{-1}\{0\}$ consider
\footnote{Again, $ (\lambda - w)^{j} $ has to be interpreted as $\lambda^{-j}$ if $w=\infty$; see \thref{unendlich}.}
\begin{align} \nonumber 
\Bigg\vert  \psi(\lambda) &- \sum_{j=0}^{{\mf d(w)-1}} \psi(w)_j (\lambda-w)^j \Bigg\vert \leq  \\ \label{langewurscht1}
&\leq \left\vert  (\phi(\lambda) - \mu)^{-1}  - \left( \sum_{j=0}^{{\mf d(w)-1}} \phi(w)_j (\lambda-w)^j \, - \mu \right)^{-1} \right\vert +  \\ \label{langewurscht2}
	  &\ + \left\vert \left( \sum_{j=0}^{{\mf d(w)-1}} \phi(w)_j (\lambda-w)^j  \, - \mu  \right)^{-1} -\sum_{j=0}^{{\mf d(w)-1}} \psi(w)_j (\lambda-w)^j \right\vert .
\end{align}

As seen above, we find a positive constant $C$ such that $|\phi(\lambda)-\mu|^{-1}\leq C$ 
for all $\lambda \in \sigma(A)\backslash q^{-1}\{0\}$. Since the denominator of the second term in 
\eqref{langewurscht1} is continuous and not zero at $\lambda = w$, 
there is a neighbourhood $U(w)$ of $w$ such that also the second term is bounded by $C$ on $U(w)$. 
Therefore,
\begin{multline*}
\left\vert  (\phi(\lambda) - \mu)^{-1}  - \left( \sum_{j=0}^{{\mf d(w)-1}} \phi(w)_j (\lambda-w)^j \, - \mu \right)^{-1} \right\vert \leq \\ \leq C^2
\left\vert  \phi(\lambda)  - \sum_{j=0}^{{\mf d(w)-1}} \phi(w)_j (\lambda-w)^j  \right\vert \leq C^2 D \, \big\vert ( \lambda - w  )^{\mf d(w)} \big\vert,
\end{multline*}
for all $\lambda \in U(w)\setminus \{w\}$, where $D>0$ is a constant which comes from \thref{zerlgeig} applied to $\phi$.

Finally, factoring out the first term and using the definition of $\psi(w)$ we estimate
\eqref{langewurscht2} from above by the following expression, where
for reasons of convenience we set $x_j=0$ for $x\in \bb C^{\mf d(w)+1}$ and $j>\mf d(w)$.
\begin{multline*}
\left\vert \left( \sum_{j=0}^{{\mf d(w)-1}} \phi(w)_j (\lambda-w)^j  \, - \mu  \right)^{-1} -\sum_{j=0}^{{\mf d(w)-1}} \psi(w)_j (\lambda-w)^j \right\vert \leq \\
\begin{array}{l}
\leq \displaystyle C \left\vert 1 - \left(\sum_{j=0}^{{\mf d(w)-1}}   \left(\phi-\mu \mathds{1}_{q,A}\right)\!(w)_j (\lambda-w)^j\right)\!\left(\sum_{j=0}^{{\mf d(w)-1}} \psi(w)_j (\lambda-w)^j\right) \right\vert  \\ \vspace*{-2mm} \\
= \displaystyle C \left\vert 1 - \sum_{j=0}^{{2 \mf d(w)-2}} \left( \sum_{i=0}^j \left(\phi-\mu \mathds{1}_{q,A}\right)\!(w)_i \, \psi(w)_{j-i}  \right)(\lambda-w)^j    \right\vert\\ \vspace*{-2mm} \\
= \displaystyle C \left\vert 1 - 1 - \sum_{j= \mf d(w)}^{{2 \mf d(w)-2}} \left( \sum_{i=0}^j \left(\phi-\mu \mathds{1}_{q,A}\right)\!(w)_i \, \psi(w)_{j-i}  \right)(\lambda-w)^j    \right\vert\\ \vspace*{-2mm} \\
\leq \displaystyle d \, \big\vert ( \lambda - w )^{\mf d(w)} \big\vert    ,
\end{array}
\end{multline*}
\noindent
for a constant $d>0$. \thref{zerlgeig} gives $\psi\in \mc F(q,A)$, and the calculation
\[
	 \left(\phi(A) - \mu I\right) \psi(A) = \big((\phi - \mu \mathds{1}_{q,A})\cdot \psi \big) (A)= 
	 \mathds{1}_{q,A}(A)=I = \psi(A) \left(\phi(A) - \mu I\right) 
\]
shows $\mu \in \rho(\phi(A))$.
\end{proof}

\begin{remark}\thlab{deltinspka}
    Assume that $\phi\in \mc M(q,A)$ vanishes on $\sigma(A)\setminus \{w\}$ and 
    $\phi(w)_{j}=0$ for $j=0,\dots,\mf d(w) - 1$, where
    $w\in q^{-1}\{0\}$. We see from \thref{zerlgeig} that then $\phi\in\mc F(q,A)$. From  
    \thref{einbett2} and \thref{Fform} we infer that then 
    $\phi=g\cdot q_{q,A}$ with $g\vert_{\sigma(A)\setminus \{w\}}=0$.
    
    If $w \not\in \sigma_p(\Theta(A))$, i.e.\ for $\Theta(A)$'s spectral measure $E$ we have $E(\{w\})=0$,
    which is in particular true for non-real $w$,
    then $\int_{\sigma(\Theta(A))} g \, dE=0$ and in turn $\phi(A)=0$.
    
    Consequently, $\phi(A)$ only depends on $\phi(w)_{\mf d(w)}$ if $w \in \sigma_p(\Theta(A))$.
\end{remark}

\begin{proposition}\thlab{projectionprop}
    With the notation and assumptions as in \thref{zwe4578}  
    let $\Delta\subseteq \sigma(A)$ be a Borel subset such that no point from 
    $q^{-1}\{0\}$ is contained in $\partial_{\sigma(A)} \Delta$.
    
    Then $\phi_\Delta$ defined by $\phi_\Delta(z) = \mathds{1}_\Delta(z)$ for 
    $z\in\sigma(A)\setminus q^{-1}\{0\}$ and by 
    $\phi_\Delta(z) = \mathds{1}_\Delta(z) (1,0,\dots,0)$ for $z\in q^{-1}\{0\}\cap\sigma(A)$
    belongs to $\mc F(q,A)$.
    $P:=\phi_\Delta(A)$ is a bounded projection in $\mc K$ with $P^+=\phi_{\overline{\Delta}}(A)$, where 
    $\overline{\Delta}=\{\bar z: z \in \Delta\}$.
    
    Moreover, we have
    $A = (P\times P)(A) \dot + ((I-P)\times (I-P))(A)$ where
    the spectrum of $(P\times P)(A)$ in $P \mc K$ satisfies
    $\sigma((P\times P)(A)) \subseteq \Cl(\Delta)$.
\end{proposition}
\begin{proof}
    As $q^{-1}\{0\} \cap \partial_{\sigma(A)} \Delta = \emptyset$ for each
    $w\in q^{-1}\{0\}\cap\sigma(A)$ we can choose a neighbourhood $O(w)$ of $w$, which is open in
    $\bb C\cup\{\infty\}$, such that $O(w)\cap\sigma(A)$ is either totally contained in 
    $\Delta$ or totally contained in $\sigma(A)\setminus \Delta$.
    Setting 
    \[
	\Omega := \Delta \cup \bigcup_{\stackrel{w \in q^{-1}\{0\}
		\cap\sigma(A)}{O(w)\cap\sigma(A)\subseteq \Delta}} O(w) \,,
    \]
    $\mathds{1}_{\Omega}$ is locally constant at each point from $q^{-1}\{0\}\cap\sigma(A)$.
    By \thref{germbddm} the function $(\mathds{1}_{\Omega})_{q,A}$ is well-defined and belongs to
    $\mc F(q,A)$. It is easy to check that $(\mathds{1}_{\Omega})_{q,A} = \phi_\Delta$.
    From $\phi_\Delta^2=\phi_\Delta$ we conclude that $P:=\phi_\Delta(A)$ is a projection.
    Since $\phi_{\Delta}^\# = \phi_{\overline{\Delta}}$, \thref{zwe4578} yields 
    $P^+=\phi_{\overline{\Delta}}(A)$.
    
    From \thref{bicommrem} and \thref{comminter} we obtain $(P\times P)(A) \subseteq A$.
    Since $P\times P$ is a projection on $\mc K\times \mc K$, this implies
    $A = (P\times P)(A) \dot + ((I-P)\times (I-P))(A)$ and in turn, 
    $(A-z)^{-1}\vert_{P\mc K} = ((P\times P)(A) - z I_{\mc K})^{-1}$
    for all $z\in \rho(A)$.
	
    In order to show $\sigma((P\times P)(A)) \subseteq \Cl(\Delta)$, fix $\eta\in \rho(A)$. 
    For $\zeta\not\in \Cl(\Delta)$ we can choose the open sets $O(w)$ from above such that 
    the closure of $O(w)$ does not contain $\zeta$ for all 
    $w \in q^{-1}\{0\} \cap\sigma(A)$ with $O(w)\cap\sigma(A)\subseteq \Delta$.
    Then the functions\footnote{The right hand sides are 
    defined to be zero for $z\not\in \Omega$.} 
    \[
	f(z) = \mathds{1}_{\Omega}(z) \cdot \frac{(z-\eta)(\zeta-\eta)}{\zeta-z} \ \ \text{ and } \ \
	g(z) = \mathds{1}_{\Omega}(z) \cdot \big(\frac{1}{z-\eta} - \frac{1}{\zeta-\eta} \big)
    \]
    are well-defined, bounded and measurable on $(\bb C\cup \{\infty\})\setminus \{\eta\}$. Moreover, they are
    locally holomorphic at all points from $q^{-1}\{0\}\cap\sigma(A)$, and their product
    is $\mathds{1}_{\Omega}(z)$. From \thref{germeigpost} (see also \thref{germeig2}) and \thref{zwe4578} we obtain  
    \[
	P = f_{q,A}(A) \, g_{q,A}(A) = f_{q,A}(A) \, \Big((A-\eta)^{-1} - \frac{1}{\zeta-\eta} \Big) P \,,
    \]
    and in turn $\frac{1}{\zeta-\eta} \in \rho((A-\eta)^{-1}\vert_{P\mc K})$.
    Therefore, $\zeta\in \rho((P\times P)(A))$.
\end{proof}

Clearly, for $\Delta=\sigma(A)$ we have $\phi_\Delta(A)=\mathds{1}_{q,A}(A)=I$.

\begin{remark}\thlab{rieszpro1}
    For $w\in \sigma(A)$ we can apply \thref{projectionprop} to $\{w\}$ and $\sigma(A)\setminus \{w\}$.
    From the spectral assertion in \thref{projectionprop} we see that 
    $\phi_{\{w\}}$ is the Riesz projection corresponding to the isolated subset $\{w\}$.
    
    For a Borel subset $\Delta\subseteq \bb R\cup\{\infty\}$, such that no point from 
    $q^{-1}\{0\}$ is contained in $\partial_{\sigma(A)} \Delta$, the projection
    $\phi_{\Delta}(A)$ is self-adjoint. In fact, these projections constitute the well known
    spectral projections for definitizable operators on Krein spaces originally found by Heinz Langer.
\end{remark}

Finally, we present a result on the existence of other definitizing rational functions.

\begin{corollary}\thlab{anderedefinitz}
  Let $A$ be a self-adjoint definitizable linear relation on a Krein space $\mc K$, 
  let $q$ be a real definitizing rational function, and consider an arbitrary $r\in \bb C_{\rho(A)}(z)$.

  If every $w\in q^{-1} \{0\} \cap \sigma(A)$ is a zero of $r$ with the same or 
  higher multiplicity, and if $\frac{r}{q}$ assumes non-negative values on $\sigma(A) \cap \bb R$, then 
  also $r$ is a definitizing rational function for $A$.
\end{corollary}
\begin{proof}
    By the first assumption we have $r_{q,A} = g\cdot q_{q,A}$ for a continuous $g: \sigma(A) \to \bb C$. By the 
    second assumption $g \geq 0$. Our functional calculus applied to this decomposition gives 
    \[
	[r(A) x, x] = \int_{\sigma(\Theta(A))} g \, d[E T^+ x , T^+ x] \geq 0 \,.
    \]
\end{proof}

\printbibliography

\end{document}